\documentclass[a4paper,11pt,twoside]{article}

\usepackage[left=2cm,right=2cm,top=2cm,bottom=2cm]{geometry}

\usepackage{amsmath,amsfonts,amssymb,amsthm,array}
\usepackage{amsmath}

\usepackage{caption}
\usepackage{subcaption}
\usepackage{algorithm}
\usepackage{algpseudocode}
\usepackage{bm}
\usepackage{color}
\usepackage{graphicx}
\usepackage{cancel}

\usepackage[algo2e]{algorithm2e} 

\newcommand{\ac}{\alpha}

\newcommand{\Exp}{\mathbb{E}}

\newcommand{\E}[1]{{\mathbb{E}\left[#1\right] }}    

\newcommand{\Prob}{\mathbb{P}}
\newcommand{\R}{\mathbb{R}}
\DeclareMathOperator{\Var}{\mathbb{V}}

\usepackage{array}
\newcolumntype{P}[1]{>{\centering\arraybackslash}p{#1}}
\newcolumntype{M}[1]{>{\centering\arraybackslash}m{#1}}

\newcommand{\eqdef}{\overset{\text{def}}{=}}

\newcommand{\cC}{{\cal C}}
\newcommand{\cD}{{\cal D}}
\newcommand{\cE}{{\cal E}}

\newcommand{\cG}{{\cal G}}

\newcommand{\cV}{{\cal V}}

\newcommand{\bA}{{\bf A}}

\newcommand{\bL}{{\bf L}}

\newcommand{\bS}{{\bf S}}

\newcommand{\mA}{{\bf A}}

\newcommand{\mI}{{\bf I}}

\newcommand{\mL}{{\bf L}}

\newcommand{\mS}{{\bf S}}

\theoremstyle{plain}
\newtheorem{thm}{Theorem}[]
\newtheorem{lem}[thm]{Lemma}

\newtheorem{defn}[thm]{Definition}
\newtheorem{cor}[thm]{Corollary}

\theoremstyle{remark}

\usepackage[colorinlistoftodos,bordercolor=orange,backgroundcolor=orange!20,linecolor=orange,textsize=scriptsize]{todonotes}

\newcommand{\algG}{1} 
\newcommand{\algE}{3} 
\newcommand{\algB}{2} 
\newcommand{\algN}{4} 

\newcommand{\ones}{{\bf 1}}

\usepackage[colorlinks=true,linkcolor=blue]{hyperref} 

\begin{document}
\title{\bf Privacy Preserving Randomized Gossip Algorithms}
\author{Filip Hanzely${}^\star$ \and Jakub Kone\v{c}n\'{y}${}^\star$  \and Nicolas Loizou${}^\star$\and
Peter Richt\'{a}rik${}^{\diamond\star}$
\and Dmitry Grishchenko${}^\dagger$\\ \phantom{xxx}\\
${}^\star$ \em University of Edinburgh, UK\\
${}^\diamond$ \em KAUST, KSA\\
${}^\dagger$ \em Higher School of Economics, Russia}

\date{June 21, 2017}
\maketitle


\begin{abstract}
In this work we present three different randomized gossip algorithms for  solving the average consensus problem while at the same time protecting  the information about the initial private values stored at the nodes. We give iteration complexity  bounds for  all methods, and perform extensive numerical experiments. 
\end{abstract}

\section{Introduction}
\label{sec:introduction}

In this paper we consider the average consensus (AC) problem. Let $\cG=(\cV,\cE)$ be an undirected connected network  with node set $\cV=\{1,2,\dots,n\}$ and edges $\cE$ such that $|\cE| = m$. Each node $i \in \cV$ ``knows'' a private value $c_i \in \R$. The goal of AC is for every node of the network to compute the average of these values, $\bar{c}\eqdef\tfrac{1}{n}\sum_i c_i$, in a distributed fashion. That is, the exchange of information can only occur between connected nodes (neighbours).

The literature on methods for solving the average consensus problem is vast and has long history \cite{tsitsiklis1984problems, tsitsiklis1986distributed, bertsekas1989parallel, kempe2003gossip}. The algorithms for solving this problem can be divided in two broad categories: the average consensus algorithms \cite{xiao2004fast} and the gossip algorithms \cite{boyd2006randomized, shah2009gossip}. The main difference is that the former work in a synchronous setting while the gossip algorithms model the case of asynchronous setting. In the average consensus algorithms, all nodes of the network update their values simultaneously by communicate with a set of their neighbours and in all iterations the same update occurs. In gossip algorithms, at each iteration, only one edge is selected randomly, and the corresponding nodes update their values to their average. In this work we focus on randomized gossip algorithms and propose techniques for protecting information of the initial values $c_i$, in the case when these may be sensitive.

In this work  we develop and analyze three private variants of the randomized pairwise gossip algorithm for solving the average consensus problem. As an additional requirement  we wish to prevent  nodes to ``learn'' information about the private values of other nodes. While we shall not formalize the notion of privacy preservation in this work, it will be intuitively clear that our methods indeed make it harder for nodes to infer information about the private values of other nodes.

\subsection{Background}

The average consensus problem and randomized gossip algorithms for solving it appear in many applications, including distributed data fusion in sensor networks \cite{xiao2005scheme}, load balancing \cite{cybenko1989dynamic} and clock synchronization \cite{freris2012fast}. This subject was first introduced in \cite{tsitsiklis1984problems}, and was studied extensively in the last decade; the seminal 2006 paper of Boyd et al.\ \cite{boyd2006randomized} on randomized gossip algorithms motivated a large amount of subsequent research and generated more than $1500$ citations to date.

In this work, we focus on modifying the basic algorithm of \cite{boyd2006randomized}, which we refer to as ``Standard Gossip'' algorithm. In the following, we review several avenues of research the gossip algorithms were evolved. While we do not address any privacy considerations in these settings, they can serve as inspiration for further work. For a survey of relevant work prior to 2010, we refer the reader to reviews in \cite{2010gossip, olfati2007consensus, ren2007information}.

The \emph{Geographic Gossip algorithm} was proposed in \cite{dimakis2008geographic}, in which the authors combine the gossip approach with a geographic routing towards a randomly chosen location with main goal the improvement of the convergence rate of Standard Gossip algorithm. In each step, a node is activated, assuming that it is aware of its geographic location and some additional assumptions on the network topology, it chooses another node from the rest of the network (not necessarily one of its neighbours) and performs a pairwise averaging with this node. Later, using the same assumptions, this algorithm was extended into \emph{Geographic Gossip Algorithm with Path Averaging} \cite{benezit2010order}, in which connected sequences of nodes were chosen in each step and they averaged their values. More recently, in \cite{freschi2016accelerating} and \cite{freschi2017} authors propose a geographic and path averaging methods which converge to the average consensus without the assumption that nodes are aware of their geographic location.

Another important class of randomized gossip algorithms are the \emph{Broadcast Gossip algorithms}, firts proposed in \cite{aysal2009broadcast} and then extended in \cite{franceschelli2011distributed, wu2013, jun2013performance}. The idea of this algorithm is simple: In each step a node in the network is activated uniformly at random, following the asynchronous time model, and broadcasts its value to its neighbours. The neighbours receive this value and update their own values. It was experimentally shown that this method converge faster than the pairwise and geographic randomized gossip algorithms.

Alternative approach to the gossip framework are so called \emph{non-randomized Gossip algorithms} \cite{mou2010deterministic, he2011periodic, liu2011deterministic, yu2017distributed}. Typically, this class of algorithms executes the pairwise exchanges between nodes in a deterministic, such as pre-defined cyclic, order. $T$-periodic gossiping is a protocol which stipulates that each node must interact with each of its neighbours exactly once every $T$ time units. Under suitable connectivity assumptions of the network $\mathcal{G}$, the $T$-periodic gossip sequence will converge at a rate determined by the magnitude of the second largest eigenvalue of the stochastic matrix determined by the sequence of pairwise exchanges which occurs over a period. It has been shown that if the underlying graph is a tree, this eigenvalue is the same for all possible $T$-periodic gossip protocols.

A different approach, uses memory in the update of the values each node holds, to get \emph{Accelerated Gossip algorithms}. The nodes update their value using an update rule that involve not only the current values of the sampled nodes but also their previous values. This idea is closely related to the shift register methods studied in numerical linear algebra for improving the convergence rate of linear system solvers. The works \cite{cao2006accelerated, 2013analysis} have shown theoretically and numerically, that under specific assumptions this idea can improve the performance of the Standard Gossip algorithm.

\emph{Randomized Kaczmarz-type Gossip algorithms.} Very recently has been proved that popular randomized Kaczmarz-type methods for solving large linear systems can also solve the AC problem. In particular, in \cite{SDA} and \cite{LoizouRichtarik} it was shown how that existing Randomized Kaczmarz and Randomized Block Kaczmarz methods can be interpreted as randomized gossip algorithms for solving the AC problem, by solving a particular system encoding the underlying network structure. This approach was the first time that a connection between the two research areas of linear system solvers and distributed algorithms have been established.

In this work we are interested in the asynchronous time model \cite{boyd2006randomized, bertsekas1989parallel}. More precisely, we assume that each node of our network has a clock which ticks at a rate of $1$ Poisson process. This is equivalent of having available a global clock which ticks according to a rate $n$ Poisson process and selects an edge of the network uniformly at random. In general the synchronous setting (all nodes update the values of their nodes simultaneously using information from a set of their neighbours) is convenient for theoretical considerations, but is not representative of some practical scenarios, such as the distributed nature of sensor networks. For more details on clock modelling we refer the reader to \cite{boyd2006randomized}, as the contribution of this paper is orthogonal to these considerations.

\emph{Privacy and Average Consensus.} 
Finally, the introduction of notions of privacy within the AC problem is relatively recent in the literature, and the existing works consider two different ideas.

The concept of differential privacy \cite{dwork2014algorithmic} is used to protect the output value $\bar{c}$ computed by all nodes in \cite{huang2012differentially}. In this work, an exponentially decaying Laplacian noise is added to the consensus computation. This notion of privacy refers to protection of the final average, and formal guarantees are provided.

A different goal is the protection of the initial values $c_i$ the nodes know at the start. In \cite{manitara2013privacy, mo2017privacy}, the goal is to make sure that each node is unable to infer a lot about the initial values $c_i$ of any other node. Both of these methods add noise correlated across individual iterations, to make sure they converge to the exact average. A formal notion of privacy breach is formulated in \cite{mo2017privacy}, in which they also show that their algorithm is optimal in that particular sense.

\subsection{Main Contributions}

In this work we present three different approaches for solving the Average Consensus problem while at the same time protecting the information about the initial values. All of the above mentioned works focus on the synchronous setting of the AC problem. This work is the first which combines the \emph{gossip framework} with the privacy concept of protection of the initial values. It is important to stress that we provide tools for protection of the initial values, but we do not address any specific notion of privacy or a threat model, nor how these quantitatively translate to any explicit measure. These would be highly application dependent, and we only provide theoretical convergence rates for the techniques we propose.

The methods we propose are all dual in nature. The dual approach is explained in detail in Section~\ref{sec:duality}. It was first proposed for solving linear systems in \cite{SDA} and then extend to the concept of average consensus problems in \cite{LoizouRichtarik}. The dual updates immediately correspond to updates to the primal variables, via an affine mapping. One of the contributions of our work is that we exactly recover existing convergence rates for the primal iterates as a special case.

We now outline the three different techniques we propose in this paper, which we refer to as ``Binary Oracle'', ``$\epsilon$-Gap Oracle'' and ``Controlled Noise Insertion''. The first two are, to the best of our knowledge, the first proposal of weakening the oracle used in the gossip framework. The latter is inspired by and similar to the addition of noise proposed in \cite{mo2017privacy} for the synchronous setting. We extend this technique by providing explicit finite time convergence guarantees.

{\bf Binary Oracle.}
The difference from standard Gossip algorithms we propose is to reduce the amount of information transmitted in each iteration to a single bit. More precisely, when an edge is selected, each corresponding node will only receive information whether the value on the other node is smaller or larger. Instead of setting the value on each node to their average, each node increases or decreases its value by a pre-specified step.

{\bf $\epsilon$-Gap Oracle.}
In this case, we have an oracle that returns one of three options, and is parametrized by $\epsilon$. If the difference in values of sampled nodes is larger than $\epsilon$, an update similar to the one in Binary Oracle is taken. Otherwise, the values remain unchanged. An advantage compared to the Binary Oracle is that this approach will converge to a certain accuracy and stop there, determined by $\epsilon$ (Binary Oracle will oscillate around optimum for a fixed stepsize). However, in general it will disclose more information about the initial values.

{\bf Controlled Noise Insertion.}
This approach is inspired by the works of \cite{manitara2013privacy, mo2017privacy}, and protects the initial values by inserting noise in the process. Broadly speaking, in each iteration, each of the sampled nodes first add a noise to its current value, and an average is computed afterwards. Convergence is guaranteed because of correlation in the noise across iterations. Each node remembers the noise it added last time it was sampled, and in the following iteration, the previously added noise is first subtracted, and a fresh noise of smaller magnitude is added.

Empirically, the protection of initial values is provided by first injecting noise in the system, which propagates across network, but is gradually withdrawn to ensure convergence to the true average.

\begin{table}[!h]
\centering
\begin{tabular}{ |p{4.4cm}||M{3.8cm}|M{4.4cm}|M{0.8cm}|  }
 \hline
 \multicolumn{4}{|c|}{Main Results} \\
 \hline
 \hbox{Randomized Gossip Methods} & Convergence Rate & Success Measure & Thm \\
 \hline
 \hline
 Standard Gossip \cite{boyd2006randomized} & $\left( 1-\frac{\ac(\cG)}{2m} \right)^k$ & $\E{\tfrac{1}{2}\|\bar{c} \ones - x^k\|^2}$ & \ref{thm:G} \\
  \hline
 \hline
 {\bf New:} Private Gossip with Binary Oracle & $1 / \sqrt{k}$ & $\min_{t \leq k} \E{\frac{1}{m}\sum_{e} |x^t_i - x^t_j|}$ & \ref{thm:jhs988sh} \\
 \hline
 {\bf New:} Private Gossip with $\epsilon$-Gap Oracle & $ 1 / (k \epsilon^2)$ & $\E{\frac{1}{k}\sum_{t=0}^{k-1}\Delta^t (\epsilon)}$ & \ref{thm:09y09s9ffs} \\
 \hline
 {\bf New:} Private Gossip with Controlled Noise Insertion & $\left( 1-\min\left( \frac{\ac(\cG)}{2m},\frac{\gamma}{m}\right) \right)^k$ & $\E{ D(y^*)- D(y^{k}) }$& \ref{T: ng general convergence} \\
 \hline
\end{tabular}
\caption{Complexity results of all proposed gossip algorithms.}
\label{table1}
\end{table}

{\bf Convergence Rates of our Methods.} In Table~\ref{table1}, we present summary of convergence guarantees for the above three techniques. By $\|\cdot\|$ we denote the standard Euclidean norm. 

The two approaches to restricting the amount of information disclosed, Binary Oracle and $\epsilon$-Gap Oracle, converge slower than the standard Gossip. In particular, these algorithms have sublinear convergence rate. At first sight, this should not be surprising, since we indeed use much less information. However, in Theorem~\ref{thm: stepsize_adaptive}, we show that if we had in a certain sense perfect global information, we could use it to construct a sequence of adaptive stepsizes, which would push the capability of the binary oracle to a linear convergence rate. However, this rate is still $m$-times slower than the standard rate of the binary gossip algorithm.  We note, however,  that having global information at hand  is an impractical assumption. Nevertheless, this result highlights that there is a potentially large scope for improvement, which we leave for future work.

These two oracles could be in practice implemented using established secure multiparty computation protocols \cite{cramer2015secure}. However, this would require the sampled nodes to exchange more than a single message in each iteration. This is inferior to the requirements of the standard gossip algorithm, but the concern of communication efficiency is orthogonal to the contribution of this work, and we do not address it further.

The approach of Controlled Noise Insertion yields a linear convergence rate which is driven by maximum of two factors. Without going into details, which of these is bigger depends on the speed by which the magnitude of the inserted noise decays. If the noise decays fast enough, we recover the convergence rate of standard the gossip algorithm. In the case of slow decay, the convergence is driven by this decay. By $\ac(\cG)$ we denote the {\em algebraic connectivity} of  graph $\cG$ \cite{fiedler1973algebraic}. The parameter $\gamma$ controls the decay speed of the inserted noise, see \eqref{Eq: phi_i def}.

\subsection{Measures of Success} Second major distinction to highlight is that convergence of each of these proposals naturally depend on a different measure of suboptimality. All of them go to $0$ as we approach optimal solution. The details of these measures will be described later in the main body, but we give a brief summary of their connections below.

The standard Gossip and Controlled Noise Insertion depend on the same quantity, but we present the latter in terms of dual values as this is what our proofs are based on. Lemma~\ref{L: rel measures} formally specifies this equivalence. The binary oracle depends on average difference among directly connected nodes. The measure for the $\epsilon$-Gap Oracle depends on quantities $\Delta^t(\epsilon)=\frac1m\left|\{ (i,j)\in  \cE: |x_i^t-x_j^t|\geq \epsilon\} \right|$, which is the number of edges, connecting nodes, values of which differ by more than $\epsilon$. 

To draw connection between these measures, we provide the following lemma, proved in the Appendix. The dual function  $D$ is formally defined in Section~\ref{sec:duality}.

\begin{lem} 
\label{L: rel measures}
(Relationship between convergence measures)
Suppose that $x$ is primal variable corresponding to the dual variable $y$ as defined in \eqref{Eq: duality mapping}.
Dual suboptimality can be expressed as the following \cite{SDA}:
\begin{equation}
\label{eq:99d8gds}
D(y^*)-D(y) = \tfrac{1}{2}\|\bar{c} \ones - x\|^2.
\end{equation}

Moreover, for any $x\in \R^{n}$ we have :

\begin{eqnarray}
 \frac{1}{2n}\sum_{i=1}^n \sum_{j=1}^n  (x_j  -x_i)^2
&=& 
\|\bar{c}\ones - x\|^2 
\label{Eq: equality}
\\
\sum_{e=(i,j)\in \cE}|x_i-x_j|
&\leq& 
 \sqrt{mn}\|\bar{c} \ones -x \|,
\label{Eq: s upper bound}
\\
\sum_{e=(i,j)\in \cE}|x_i-x_j|
&\geq& 
\sqrt{\ac(\cG)}\|\bar{c} \ones-x \|,
\label{Eq: s lower bound}
\\
\sum_{e=(i,j)\in \cE}|x_i-x_j|
&\geq& \epsilon
  \left|\{ (i,j)\in  \cE: |x_i-x_j|\geq\epsilon\}\right|.
\label{Eq: delta bound}
\end{eqnarray}

\end{lem}

\subsection{Outline}

The remainder of this paper is organized as follows: Section~\ref{sec:duality} introduces  the basic setup that are used through the paper. A detailed explanation of the duality behind the randomized pairwise gossip algorithm is given. We also include a novel and insightful dual analysis of this method as it will make it easier to the reader to parse later development.   In Section~\ref{sec:Private} we present our three private gossip algorithms as well as the associated iteration complexity results. Section~\ref{sec:experiments} is devoted to the numerical evaluation of our methods. Finally, conclusions are drawn in Section~\ref{sec:conclusion}. All proofs not included in the main text can be found  in the Appendix.

\section{Dual Analysis of Randomized Pairwise Gossip} \label{sec:duality}

As we outlined in the introduction, our approach to extending the (standard) randomized pairwise  gossip algorithm to privacy preserving variants  utilizes duality. The purpose of this section is to formalize this duality, following the development in \cite{SDA}. In addition, we provide a novel and self-contain dual analysis of randomized pairwise gossip. While this is of an independent interest, we include the proofs as their understanding aids in the understanding of the more involved proofs of our private gossip algorithms developed in the remainder of the paper.

\subsection{Primal and Dual Problems}
\label{sec:primal_dual}

Consider solving the (primal) problem of projecting a given vector $c=x^0\in \R^n$ onto the solution space of a  linear system: 
\begin{equation}\label{eq:primal}\min_{x\in \R^n} \{P(x) \eqdef \tfrac{1}{2}\|x-x^0\|^2\} \quad \text{subject to} \quad \bA x=b,\end{equation} 
where $\bA\in \R^{m\times n}$, $b\in \R^m$, $x^0\in \R^n$.
We assume the problem is feasible, i.e., that the system $\mA x = b$ is consistent. With the above optimization problem we associate the dual problem
\begin{equation}\label{eq:dual}\max_{y\in \R^m} D(y)\eqdef (b-\bA x^0)^\top y - \tfrac{1}{2}\|\bA^\top y\|^2.\end{equation}
The dual is an unconstrained concave (but not necessarily strongly concave) quadratic maximization problem. It can be seen that as soon as the system $\mA x = b$ is feasible, the dual problem is bounded. Moreover, all bounded concave quadratics in $\R^m$ can be written in the as $D(y)$ for some matrix $\mA$ and vectors $b$ and $x^0$ (up to an additive constant).

With any dual vector $y$ we associate the primal vector  via an affine transformation:
\[x(y) = x^0 + \mA^\top y.\]
It can be shown that if $y^*$ is dual optimal, then $x^*=x(y^*)$ is primal optimal. Hence, any dual algorithm producing a sequence of dual variables $y^t \to y^*$ gives rise to a corresponding primal algorithm producing the sequence $x^t \eqdef x(y^t) \to x^*$. We shall now consider one such dual algorithm.

\subsection{Stochastic Dual Subspace Ascent}

SDSA~\cite{SDA} is a stochastic method for  solving the dual problem \eqref{eq:dual}. If we assume that $b=0$, the iterates of SDSA take the form
\begin{equation}\label{eq:SDSA}y^{t+1} = y^t - \bS_t(\bS_t^\top \bA \bA^\top \bS_t)^\dagger \bS_t^\top \bA(x^0 + \bA^\top y^t),\end{equation}
where $\mS_t$ is a random matrix drawn from independently at each iteration $t$ from an arbitrary but fixed distribution  $\cD$, and $\dagger$ denotes the Moore-Penrose pseudoinverse.

The corresponding primal iterates are defined via:
\begin{equation}
x^{t} \eqdef x(y^t) = x^0 + \bA^\top y^t
\label{Eq: duality mapping}
\end{equation}

The relevance of this all to average consensus follows through the observation, as we shall see next, that for a specific choice of matrix $\mA$ (as defined in the next subsection) and distribution $\cD$, method \eqref{Eq: duality mapping} is equivalent to the (standard) randomized pairwise gossip method. In that case, SDSA is a dual variant of randomized pairwise gossip. In particular, we define $\cD$ as follows: $S_t$ is a  unit basis vector in $\R^m$, chosen uniformly at random from the collection of all such unit basis vectors, denoted $\{f_e \;|\; e\in \cE\}$. In this case, SDSA is a randomized coordinate ascent method applied to the dual problem.

For general distributions $\cD$, the primal methods obtained from SDSA via \eqref{Eq: duality mapping} (but without observing that it arises that way) was fist proposed and studied in \cite{SIMAX2015} under a full rank assumption on $\mA$. This assumption was lifted, and duality exposed and studied as we explain it here, in \cite{SDA}. For deeper insights and connections to stochastic optimization, stochastic fixed point methods, stochastic linear systems and probabilistic intersection problems, we refer the reader to \cite{ASDA}. The method can be extended to compute the inverse \cite{inverse} and pseudoinverse \cite{pseudoinverse} of a matrix, in which case it has deep connections with quasi-Newton updates \cite{inverse}. In particular, it can be used to design a stochastic block extension of the famous BFGS method \cite{inverse} and applied to the empirical risk minimization problem arising in machine learning to design a fast stochastic quasi-Newton training method \cite{SBFGS}.

\subsection{Randomized Gossip Setup: Choosing $\mA$}

We wish $(\mA,b)$ to be an {\em average consensus (AC)} system, defined next. 

\begin{defn} (\cite{LoizouRichtarik}) 
Let  $\cG = (\cV, \cE)$ be an undirected graph with $|\cV| = n$ and $|\cE|=m$. Let $\bA$ be a real matrix with $n$ columns.
The linear system $\bA x = b$ is an ``average consensus (AC) system'' for graph $\cG$ if $\bA x = b$ iff $x_i = x_j$ for all $(i,j) \in \cE$.
\end{defn}

Note that if $\mA x = b$ is an AC system, then the solution of the primal problem \eqref{eq:primal} is necessarily
$x^* = \bar{c} \cdot \ones$, where $\bar{c} = \frac{1}{n}\sum_{i=1}^n x^0_i.$
 This is exactly what we want: we want the solution of the primal problem to be $x^*_i = \bar{c}$ for all $i$: the average of the private values stored at the nodes. 

It is easy to see that a linear system is an AC system precisely when $b=0$ and the nullspace of $\bA$ is $\{t \ones : t\in \R\}$, where $\ones$ is the vector of all ones in $\R^n$.  Hence, $\bA$ has rank $n-1$.

In the rest of this paper we focus on a specific AC system; one  in which the matrix $\bA$  is the incidence matrix of the graph $\cG$ (see Model 1 in \cite{SDA}).  In particular, we let $\mA\in \R^{m\times n}$ be the matrix defined as follows. Row $e=(i,j)\in \cE$  of $\mA$ is given by $\mA_{ei} = 1$, $\mA_{ej}=-1$ and $\mA_{el}=0$ if $l\notin \{i,j\}$.  Notice that the system $
\mA x=0$ encodes the constraints $x_i=x_j$ for all $(i,j)\in \cE$, as desired.

\subsection{Randomized Pairwise Gossip}
\label{sec:gossip}

We provide both primal and dual form of the (standard) randomized pairwise gossip algorithm. 

The primal form is standard and needs no lengthy commentary. At the beginning of the process, node $i$ contains private information $c_i = x^0_i$. In each iteration we sample a pair of connected nodes $(i,j)\in \cE$ uniformly at random, and update $x_i$ and $x_j$ to their average. We let the values at the remaining nodes intact.

\setcounter{algorithm}{\algG-1}
\begin{algorithm}[H]
\textbf{Input: }{vector of private values $c\in \R^n$}\\
{\textbf {Initialize:}} Set $x^0=c$.\\
\For {$t= 0,1,\dots, k-1$} {
	 \begin{enumerate}
\item Choose node $e = (i,j)\in \cE$ uniformly at random
\item Update the primal variable: 
  \[x^{t+1}_l = \begin{cases}
\frac{x^t_i+x^t_j}{2},\quad & l \in \{i,j\}\\
 x^{t}_l, \quad & l \notin \{i,j\}.
 \end{cases}
  \]
\end{enumerate}
 }
\textbf{return} $x^k$
\caption{(Primal form)}
\end{algorithm}

The dual form of the standard randomized pairwise gossip method is a specific instance of SDSA, as described in \eqref{eq:SDSA}, with $x^0=c$ and $\mS_t$ being a randomly chosen standard unit basis vector $f_e$ in $\R^m$ ($e$ is a randomly selected edge). It can be seen \cite{SDA} that in that case, \eqref{eq:SDSA} takes the following form:

\setcounter{algorithm}{\algG-1}
\begin{algorithm}[H]
\textbf{Input: }{vector of private values $c\in \R^n$}\\
{\textbf {Initialize:}} Set $y^0=0\in\R^m$.\\
\For {$t= 0,1,\dots, k-1$} {
\begin{enumerate}
\item Choose node $e = (i,j)\in \cE$ uniformly at random
\item Update the dual variable: 
\[y^{t+1} =  y^t + \lambda^t f_e \quad \mathrm{where} \quad \lambda^t =\mathrm{argmax}_{\lambda'} D(y^t + \lambda' f_e). \]
\end{enumerate}
 }
\textbf{return} $y^k$
\caption{(Dual form)}
\end{algorithm}

The following lemma is useful for the analysis of all our methods.It describes the increase in the dual function value after an arbitrary change to a single dual variable $e$.

\begin{lem} \label{lem:98y98yss}
Define $z=y^t + \lambda f_e$, where $e=(i,j)$ and $\lambda\in \R$. Then
 \begin{equation} \label{eq:89g9s8guffxx} D(z) - D(y^t) = -\lambda(x^t_i - x^t_j) - \lambda^2.\end{equation}
\end{lem}
\begin{proof}
The claim follows by direct calculation:
\begin{eqnarray*}
 D(y^{t}+ \lambda f_e) - D(y^t)
&=& -(\bA c)^\top (y^t + \lambda f_e) - \tfrac{1}{2}\|\bA^\top (y^t+ \lambda f_e)\|^2 + (\bA c)^\top y^t + \tfrac{1}{2}\| \bA^\top y^t\|^2\\
&=&  - \lambda f_e^\top \bA\underbrace{(c +  \bA^\top y^t)}_{x^t} - \tfrac{1}{2}\lambda^2 \underbrace{\| \bA^\top f_e\|^2}_{=2}  \quad = \quad  -\lambda (x^t_i-x^t_j) - \lambda^2.
\end{eqnarray*}
\end{proof}

The maximizer in $\lambda$ of the expression in \eqref{eq:89g9s8guffxx} leads to the exact line search formula
$\lambda^t = (x_j^t-x_i^t)/2$
used in the dual form of the method. 

\subsection{Complexity Results}

With graph $\cG = \{\cV,\cE\}$ we now associate a certain quantity, which we shall denote $\beta = \beta(\cG)$. It is the smallest nonnegative number $\beta$ such that the following inequality\footnote{We write $\sum_{(i,j)}$ to indicate sum over all {\em unordered} pairs of vertices. That is, we do not count $(i,j)$ and $(j,i)$ separately, only once. By $\sum_{(i,j)\in \cE}$ we denote a sum over all edges of $\cG$. On the other hand, by writing $\sum_i \sum_j $, we are summing over all (unordered) pairs of vertices twice.} holds for all $x\in \R^n$:

\begin{equation}\label{eq:hdgugvej}
 \sum_{(i,j)} (x_j  -x_i)^2 \leq \beta \sum_{(i,j)\in \cE} (x_j  -x_i)^2.
\end{equation}

The Laplacian matrix of graph $\cG$ is given by $\mL = \mA^\top \mA$. Let $\lambda_1(\mL)\geq \lambda_2(\mL) \geq \dots \geq \lambda_{n-1}(\mL)\geq \lambda_n(\mL)$ be the eigenvalues of $\mL$. The {\em algebraic connectivity} of $\cG$ is the second smallest eigenvalue of $\mL$:
\begin{equation} \label{eq:algebraic_connectivity} \ac(\cG) = \lambda_{n-1}(\mL).\end{equation}
We have $\lambda_{n}(\mL)=0$. Since we assume $\cG$ to be connected, we have $\ac(\cG)>0$. 
Thus,  $\ac(\cG)$ is the smallest nonzero eigenvalue of the Laplacian: $\ac(\cG)=\lambda_{\min}^+(\mL) = \lambda_{\min}^+(\mA^\top \mA ).$
As the next result states, the quantities $\beta(\cG)$ and $\ac(\cG)$ are inversely proportional.

\begin{lem}\label{lem:beta} $\beta(\cG) = \frac{n}{\ac(\cG)}.$
\end{lem}

The following theorem gives  a complexity result for (standard) randomized gossip. Our analysis is dual in nature (see the Appendix).

\begin{thm}\label{thm:G} Consider the randomized gossip algorithm (Algorithm~\algG) with uniform edge-selection probabilities: $p_e=1/m$. Then:
\[\E{D(y^*) - D(y^{k}) } \leq \left(1-\frac{\ac(\cG)}{2m }\right)^k[D(y^*) - D(y^{0}) ]. \]
\end{thm}

Theorem~\ref{thm:G}  yields the complexity estimate
${\cal O}\left(\frac{2m}{\ac(\cG)} \log(1/\epsilon)\right)$,
which exactly matches the complexity result obtained from the primal analysis~\cite{SDA}. Hence, the primal and dual analyses give the same rate. 

Randomized coordinate descent methods were first analyzed in \cite{Leventhal:2008:RMLC, Nesterov:2010RCDM, UCDC, PCDM}. For a recent treatment, see \cite{ALPHA, ESO}. Duality in randomized coordinate descent methods was studied in \cite{Shalev-Shwartz2012, QUARTZ}. Acceleration was studied in \cite{lee2013efficient, APPROX, allen2016even}. These methods extend to nonsmooth problems of various flavours \cite{SPCDM, SSP}.

With all of this preparation, we are now ready to formulate and analyze our private gossip algorithms; we do so in Section~\ref{sec:Private}.

\section{Private Gossip Algorithms} \label{sec:Private}

In this section we introduce three novel private gossip algorithms, complete with iteration complexity guarantees. In Section~\ref{sec:B} we protect privacy via a  binary  communication protocol. In Section~\ref{sec:E} we communicate more: besides binary information, we allow for the communication of a bound on the gap, introducing the $\epsilon$-gap oracle. In Section~\ref{sec: noise} we introduce a privacy-protection mechanism based on a procedure we call {\em controlled noise insertion}.

\subsection{Private Gossip via Binary Oracle} 
\label{sec:B}

We now present the gossip algorithm with Binary Oracle in detail and provide theoretical convergence guarantee. The information exchanged between sampled nodes is constrained to a single bit, describing which of the nodes has the higher value. As mentioned earlier, we only present the conceptual idea, not how exactly would the oracle be implemented within a secure multiparty protocol between participating nodes \cite{cramer2015secure}.

We will first introduce dual version of the algorithm.

\setcounter{algorithm}{\algB-1}
\begin{algorithm}[H]
\textbf{Input: }{vector of private values $c\in \R^n$, sequence of positive stepsizes $\{\lambda^t\}_{t=0}^{\infty}$}\\
{\textbf {Initialize:}} Set $y^0=0\in \R^m$, $x^0=c$ \\
\For {$t= 0,1,\dots, k-1$} {
\begin{enumerate}
\item Choose node $e = (i,j)\in \cE$ uniformly at random
\item Update the dual variable: \[y^{t+1} = \begin{cases} y^t + \lambda^t f_e, &\quad  x^t_i< x^t_j,\\
 y^t - \lambda^t f_e, & \quad x^t_i\geq x^t_j.
\end{cases}
 \] 

\item Set 
\begin{eqnarray*}
x^{t+1}_i &=& \begin{cases} x^t_i + \lambda^t, &\quad  x^t_i < x^t_j,\\
 x^t_i - \lambda^t , & \quad x^t_i\geq x^t_j.
\end{cases}
\\ 
 x^{t+1}_j &=& \begin{cases} x^t_j - \lambda^t, &\quad  x^t_i < x^t_j,\\
 x^t_j + \lambda^t  & \quad x^t_i\geq x^t_j.
\end{cases} 
\\
  x^{t+1}_l&=&x^t_l \qquad l\not\in\{i,j\}
\end{eqnarray*} 
\end{enumerate}
 }
\textbf{return} $y^k$
\caption{(Dual form)}
\end{algorithm}

The update of primal variables above is equivalent to set $x^{t+1}$ as primal point corresponding to dual iterate: $x^{t+1} = c+ \bA^\top y^{t+1} = x^t + \bA^\top (y^{t+1}-y^t)$. In other words, the primal iterates $\{x^t\}$ associated with the dual iterates $\{y^t\}$ can be written in the form:
\[x^{t+1} = \begin{cases} x^t + \lambda^t \bA_{e:}^\top, &\quad  x^t_i < x^t_j,\\
 x^t - \lambda^t \bA_{e:}^\top, & \quad x^t_i\geq x^t_j.
\end{cases} 
 \] 
It is easy to verify that due to the structure of $\bA$, this is equivalent to the updates above.

Since the evolution of dual variables $\{y^k\}$ serves only the purpose of the analysis, the method can be written in the primal-only form as follows:

\setcounter{algorithm}{\algB-1}
\begin{algorithm}[H]
\textbf{Input: }{vector of private values $c\in \R^n$, sequence of positive stepsizes $\{\lambda^t\}_{t=0}^{\infty}$}\\
{\textbf {Initialize:}} Set $x^0=c$ \\
\For {$t= 0,1,\dots, k-1$} {
\begin{enumerate}
\item Choose node $e = (i,j)\in \cE$ uniformly at random

\item Set 
\begin{eqnarray*}
x^{t+1}_i &=& \begin{cases} x^t_i + \lambda^t, &\quad  x^t_i < x^t_j,\\
 x^t_i - \lambda^t , & \quad x^t_i\geq x^t_j.
\end{cases}
\\ 
 x^{t+1}_j &=& \begin{cases} x^t_j - \lambda^t, &\quad  x^t_i < x^t_j,\\
 x^t_j + \lambda^t  & \quad x^t_i\geq x^t_j.
\end{cases} 
\\
  x^{t+1}_l&=&x^t_l \qquad l\not\in\{i,j\}
\end{eqnarray*}
\end{enumerate}
 }
\textbf{return} $x^k$
\caption{(Primal form)}
\end{algorithm}

Given a sequence of stepsizes $\{\lambda^t\}$, it will be convenient to define $\alpha^k\eqdef \sum_{t=0}^{k}\lambda^t$ and $\beta^k \eqdef \sum_{t=0}^{k}\left(\lambda^t\right)^2$.
In the following theorem, we study the convergence of the quantity
\begin{equation} \label{eq: L def}
L^t\eqdef \frac{1}{m}\sum_{e=(i,j)\in \cE} |x^t_i - x^t_j|.
\end{equation} 

\begin{thm}\label{thm:jhs988sh}  For all $k\geq 1$ we have
\begin{equation}
\min_{t=0,1,\dots,k} \E{L^t} \leq \sum_{t=0}^{k} \frac{\lambda^t}{\alpha^k}\E{L^t} \leq U^k \eqdef \frac{D(y^*)-D(y^0)}{\alpha^k} + \frac{\beta^k}{\alpha^k}.\label{Eq: binary general rate}
\end{equation}
Moreover:
\begin{enumerate}
\item[(i)] If we set $\lambda^t = \lambda^0>0$ for all $t$, then  $U^k = \frac{D(y^*)-D(y^0)}{\lambda^0 (k+1)} + \lambda^0$.
\item[(ii)] Let $R$ be any constant such that $R\geq D(y^*)-D(y^0)$. If we fix  $k\geq 1$, then the choice of stepsizes $\{\lambda^0,\dots,\lambda^k\}$ which minimizes $U^k$ correspond to the constant stepsize rule $\lambda^t = \sqrt{\tfrac{R}{k+1}}$ for all $t=0,1,\dots,k$, and $U^k = 2\sqrt{\tfrac{R}{k+1}}$.
\item[(iii)] If we set $^t=a/\sqrt{t+1}$ for all $t=0,1,\dots,k$, then 
\begin{equation*}
U^k\leq \frac{D(y^*)-D(y^0)+a^2 \left(\log(k+3/2)+\log(2) \right)}{2a\left( \sqrt{k+2} -1\right)}=O\left(\frac{\log(k)}{\sqrt{k}}\right)
\end{equation*}
\end{enumerate}
\end{thm}

The part \textit{(ii)} of Theorem~\ref{thm:jhs988sh} is useful in the case if we know exactly the number of iterations before running the algorithm, providing in a sense optimal stepsizes and rate $O(1/\sqrt{k})$. However, this might not be the case in practice. Therefore part \textit{(iii)} is also relevant, which yields the rate $O(\log(k)/\sqrt{k})$. These bounds are significantly weaker than the standard bound in Theorem~\ref{thm:G}. This should not be surprising though, as we use significantly less information than the Standard Gossip algorithm.

Nevertheless, there is a potential gap in terms of what rate can be practically achievable. The following theorem can be seen as a form of a bound on what convergence rate is possible to attain with the Binary Oracle. However, this is attained with access to very strong information needed to set the sequence of stepsizes $\lambda^t$, likely unrealistic in any application. This result points at a gap in the analysis which we leave open. We do not know whether the sublinear convergence rate in Theorem~\ref{thm:jhs988sh} is necessary or improvable without additional information about the system.

\begin{thm}
\label{thm: stepsize_adaptive}
For Algorithm~\algB\ with stepsizes chosen in iteration $t$ adaptively to the current values of $x^t$ as $\lambda^t = \frac{1}{2m}\sum_{e\in \cE}|x_i^t-x_j^t|$, we have
\begin{equation*}
\E{\|\overline{c}\ones-x^{k} \|^2}
\leq
\left(1-\frac{\ac(\cG)}{2m^2}\right)^k\|\overline{c}\ones-x^{0} \|^2
\end{equation*}
\end{thm}
Comparing Theorem~\ref{thm: stepsize_adaptive} with the result for standard Gossip in Theorem~\ref{thm:G}, the convergence rate is worse by factor of $m$, which is the price we pay for the weaker oracle.

An alternative to choosing adaptive stepsizes is the use of adaptive probabilities \cite{AdaSDCA}. We leave such a study to future work.

\subsection{Private Gossip via  $\epsilon$-Gap Oracle} 
\label{sec:E}

Here we present the gossip algorithm with $\epsilon$-Gap Oracle in detail and provide theoretical convergence guarantees. The information exchanged between sampled nodes is restricted to be one of three cases, based on difference in values on sampled nodes. As mentioned earlier, we only present the conceptual idea, not how exactly would the oracle be implemented within a secure multiparty protocol between participating nodes \cite{cramer2015secure}.

We will first introduce dual version of the algorithm.

\setcounter{algorithm}{\algE-1}
\begin{algorithm}[H]
\textbf{Input: }{vector of private values $c\in \R^n$; error tolerance $\epsilon>0$}\\
{\textbf {Initialize:}} Set $y^0 = 0 \in \R^m$; $x^0=c$. \\
\For {$t= 0,1,\dots, k-1$} {
\begin{enumerate}
\item Choose node $e = (i,j)\in \cE$ uniformly at random
\item Update the dual variable: \[y^{t+1} = \begin{cases} y^t +\frac{ \epsilon}{2} f_e, &\quad  x^t_i-x^t_j < -\epsilon\\
 y^t - \frac{ \epsilon}{2} f_e, & \quad x^t_j-x^t_i < -\epsilon,\\
y^t, & \quad \text{otherwise.}
\end{cases}
 \] 

\item  If $x^t_i \leq x^t_j -  \epsilon$ then $x^{t+1}_i = x^t_i + \frac{ \epsilon}{2}$ and $x^{t+1}_j = x^t_j - \frac{ \epsilon}{2}$ 
\item  If $x^t_j \leq x^t_i -  \epsilon$ then $x^{t+1}_i = x^t_i - \frac{ \epsilon}{2}$ and $x^{t+1}_j = x^t_j + \frac{ \epsilon}{2}$ 

\end{enumerate}
\label{GapOracle}
}
\textbf{return} $x^k$
\caption{(Dual form)}
\end{algorithm}

Note that the primal iterates $\{x^t\}$ associated with the dual iterates $\{y^t\}$ can be written in the form:
$$
x^{t+1} = \begin{cases} x^t + \frac{ \epsilon}{2} \bA_{e:}^\top, &\quad  x^t_i-x^t_j < -\epsilon\\
 x^t - \frac{ \epsilon}{2} \bA_{e:}^\top, & \quad x^t_j-x^t_i < -\epsilon,\\
x^t, & \quad \text{otherwise.}
\end{cases}
$$
The above is equivalent to setting
$x^{t+1} = x^t + \bA^\top (y^{t+1}-y^t)=c+ \bA^\top y^{t+1}$. 

Since the evolution of dual variables $\{y^t\}$ serves only the purpose of the analysis, the method can be written in the primal-only form as follows:

\setcounter{algorithm}{\algE-1}
\begin{algorithm}[H]
\textbf{Input: }{vector of private values $c\in \R^n$; error tolerance $\epsilon>0$}\\
{\textbf {Initialize:}} Set $x^0=c$. \\
\For {$t= 0,1,\dots, k-1$} {
\begin{enumerate}
\item Set $x^{t+1} = x^t$
\item Choose node $e = (i,j)\in \cE$ uniformly at random
\item  If $x^t_i \leq x^t_j -  \epsilon$ then $x^{t+1}_i = x^t_i + \frac{ \epsilon}{2}$ and $x^{t+1}_j = x^t_j - \frac{ \epsilon}{2}$ 
\item  If $x^t_j \leq x^t_i -  \epsilon$ then $x^{t+1}_i = x^t_i - \frac{ \epsilon}{2}$ and $x^{t+1}_j = x^t_j + \frac{ \epsilon}{2}$ 
\end{enumerate}
 }
\textbf{return} $x^k$
\caption{(Primal form)}
\end{algorithm}

Before stating the convergence result, let us define a quantity the convergence will naturally depend on. For each edge $e=(i,j)\in \cE$ and iteration $t \geq 0$ define the random variable \[\Delta^t_e(\epsilon) \quad \eqdef \quad \begin{cases} 1, \qquad |x^t_i-x^t_j|\geq \epsilon,\\
0, \qquad \text{otherwise.}\end{cases}\]
Moreover, let 
\begin{equation}
\label{eq:delta_k}
\Delta^t (\epsilon)\quad \eqdef \quad \frac{1}{m}\sum_{e\in \cE} \Delta^t_e(\epsilon).
\end{equation}

The following Lemma bounds the expected increase in dual function value in each iteration.

\begin{lem}\label{lem:09ys09y9ss} For all $t \geq 0$ we have
$\E{D(y^{t+1})-D(y^t) } \geq \frac{\epsilon^2}{4} \E{\Delta^t(\epsilon)}.$
\end{lem}

Our complexity result will be expressed in terms of the quantity:
\begin{equation} \label{eq: delta def}
\delta^k(\epsilon)\quad \eqdef \quad  \E{\frac{1}{k}\sum_{t=0}^{k-1}\Delta^t(\epsilon)} = \frac{1}{k}\sum_{t=0}^{k-1}\E{\Delta^t (\epsilon)}.
\end{equation}

\begin{thm}\label{thm:09y09s9ffs} For all $k\geq 1$ we have \[\delta^k(\epsilon) \leq \frac{4\left(D(y^*)-D(y^0)\right)}{k\epsilon^2}.\]
\end{thm}

Note that if $\Delta^k(\epsilon)=0$, it does not mean the primal iterate $x^k$ is optimal. This only implies that the values of all pairs of directly connected nodes are differ by less than $\epsilon$.

\subsection{Private Gossip via Controlled Noise Insertion}
\label{sec: noise}

In this section, we present the Gossip algorithm with Controlled Noise Insertion. As mentioned in the introduction, the approach is similar the technique proposed in \cite{manitara2013privacy, mo2017privacy}. Those works, however, address only algorithms in the synchronous setting, while our work is the first to use this idea in the asynchronous setting. Unlike the above, we provide finite time convergence guarantees and allow each node to add the noise differently, which yields a stronger result.

In our approach, each node adds noise to the computation independently of all other nodes. However, the noise added is correlated between iterations for each node. We assume that every node owns two parameters --- initial magnitude of the generated noise $\sigma_i^2$ and rate of decay of the noise $\phi_i$. The node inserts noise $w_i^{t_i}$ to the system every time that an edge corresponding to the node was chosen, where variable $t_i$ carries an information how many times the noise was added to the system in the past by node $i$. Thus, if we denote by $t$ the current number of iterations, we have $\sum_{i=1}^nt_i = 2t$.

In order to ensure convergence to the optimal solution, we need to choose a specific structure of the noise in order to guarantee the mean of the values $x_i$ converges to the initial mean. In particular, in each iteration a node $i$ is selected, we subtract the noise that was added last time, and add a fresh noise with smaller magnitude:
\[
w_i^{t_i} = \phi_i^{t_i}v_i^{t_i}-\phi_i^{t_i-1}v_i^{t_i-1},
\] 
where $0 \leq \phi_i<1, v_i^{-1} = 0$ and $v_i^{t_i}\sim N(0,\sigma_i^2)$ for all iteration counters $k_i \geq 0$ is independent to all other randomness in the algorithm. This ensures that all noise added initially is gradually withdrawn from the whole network.

After the addition of noise, a standard Gossip update is made, which sets the values of sampled nodes to their average. Hence, we have
\begin{eqnarray*}
\lim_{t\rightarrow \infty}
\E{\left(
\overline{c}-
\frac1n\sum_{i=1}^n x_i^t
\right)^2}
&=&
\lim_{t\rightarrow \infty}
\E{\left(
\frac1n\sum_{i=1}^n \phi_i^{t_i-1}v_i^{t_i-1}
\right)^2}
\quad
\leq
\quad 
\lim_{t\rightarrow \infty}
\E{\frac1n \sum_{i=1}^n \left(
 \phi_i^{t_i-1}v_i^{t_i-1}
\right)^2}
\\
&=&
\frac1n\lim_{t\rightarrow \infty}
\sum_{i=1}^n \E{  \left(
 \phi_i^{t_i-1}v_i^{t_i-1}
\right)^2}
\quad
=
\quad
\frac1n\lim_{t\rightarrow \infty}
\sum_{i=1}^n \E{  
 \phi_i^{2t_i-2}}\E{\left(v_i^{t_i-1}\right)^2
}
\\
&=&
\frac1n\lim_{t\rightarrow \infty}
\sum_{i=1}^n \E{  
 \phi_i^{2t_i-2}}\sigma_i^2
 \quad
 =
 \quad
 \frac1n \sum_{i=1}^n \sigma_i^2 \lim_{t\rightarrow \infty}
\E{  
 \phi_i^{2t_i-2}}
 \\
 &=&
 0,
\end{eqnarray*}
as desired. 

It is not the purpose of this paper to define any quantifiable notion of protection of the initial values formally. However, we note that it is likely the case that the protection of private value $c_i$ will be stronger for bigger $\sigma_i$ and for $\phi_i$ closer to $1$.

For simplicity, we provide only the primal algorithm below.

\setcounter{algorithm}{\algN-1}
\begin{algorithm}[H]
\textbf{Input: }{vector of private values $c\in \R^n$; initial variances $\sigma^2_i \in \R_+$ and variance decrease rate $\phi_i$ such that $0\leq \phi_i < 1$ for all nodes $i$.}\\
{\textbf{Initialize}:} Set $x^0=c$; $t_1=t_2=\dots = t_n=0$, $v_1^{-1}=v_2^{-1}=\dots = v_n^{-1}=0$.\\
\For {$t= 0,1,\dots k-1$} {
	 \begin{enumerate}
\item Choose node $e = (i,j)\in \cE$ uniformly at random
\item Generate $v_i^{t_i}\sim N(0,\sigma^2_i)$ and $v_j^{t_j}\sim N(0,\sigma^2_j)$
\item Set \[w_i^{t_i} = 
 \phi_i^{t_i}v_i^{t_i}-\phi_i^{t_i-1}v_i^{t_i-1} 
 \] 
 \[w_j^{t_j} = 
 \phi_j^{t_j}v_j^{t_j}-\phi_j^{t_j-1}v_j^{t_j-1} 
 \] 
\item Update the primal variable: \[x^{t+1}_i =x^{t+1}_j= 
\frac{x^t_i+w^{t_i}_i+x^t_j+w^{t_j}_j}{2},\ \forall\, l \neq i,j:\,x^{t+1}_l=x^{t}_l
  \]
\item Set $t_i=t_i+1$, $t_j=t_j+1$
\end{enumerate}
 }
\textbf{return} $x^k$
\caption{(Primal form)}
\end{algorithm}

We now provide results of dual analysis of Algorithm~\algN. The following lemma provides us the expected decrease in dual suboptimality for each iteration.

\begin{lem}\label{L: noise exp iteration bound}
Let $d_i$ denote the number of neighbours of node $i$. Then, 
\begin{eqnarray}
\begin{split}
\E{ D(y^*)- D(y^{t+1}) } \leq & 
\left( 1-\frac{\ac(\cG)}{2m}\right)\E{D(y^*)- D(y^{t})} +\frac{1}{4m}\sum_{i=1}^n d_i\sigma^2_i \E{\phi_i^{2t_i}}
\\& \qquad-
\frac{1}{2m}\sum_{e\in \cE} \E{\left( \phi_i^{t_i-1}v_i^{t_i-1} x_j^t+ \phi_j^{t_j-1}v_j^{t_j-1} x_i^t
\right)}.
\end{split}
\label{Eq: noise gossip iteration bound final}
\end{eqnarray}
\end{lem}

We use the lemma to prove our main result, in which we show linear convergence for the algorithm. For notational simplicity, we decided to have $\rho^t=(\rho)^t$, i.e. superscript of $\rho$ denotes its power, not an iteration counter.

\begin{thm}
Let us define the following quantities:
\begin{eqnarray*}
\rho&\eqdef& 1-\frac{\ac(\cG)}{2m},
\\
\psi^t&\eqdef&\frac{1}{\sum_{i=1}^n\left(d_i\sigma_i^2\right)}\sum_{i=1}^n d_i \sigma_i^2\left(1-\frac{d_i}{m}\left(1-\phi_i^2\right) \right)^{t}.
\end{eqnarray*}

Then for all $k\geq 1$ we have the following bound
\begin{equation*}
\E{ D(y^*)- D(y^{k}) } \leq  \rho^k \left( D(y^*)- D(y^{0}) \right)  + \frac{\sum\left(d_i\sigma_i^2\right)}{4m}\sum_{t=1}^k \rho^{k-t}\psi^{t}.
\end{equation*}
\label{T: ng general convergence}
\end{thm}

Note that $\psi^t$ is a weighted sum of $t$-th powers of real numbers smaller than one.  For large enough $t$, this quantity will depend  on the largest of these numbers. This brings us to define $M$ as the set of indices $i$ for which the quantity $1-\frac{d_i}{m}\left(1-\phi_i^2\right)$ is maximized:  
$$
M=\arg\max_{i} \left\{ 1-\frac{d_i}{m}\left(1-\phi_i^2\right)\right\}. 
$$
Then for any $i_\mathrm{max}\in M$ we have
$$
\psi^t  
\approx
 \frac{1}{\sum_{i=1}^n \left(d_i\sigma_i^2\right)} \sum_{i\in M} d_i \sigma_i^2\left(1-\frac{d_i}{m}\left(1-\phi_i^2\right) \right)^{t}
 =
  \frac{\sum_{i\in M} d_i \sigma_i^2}{\sum_{i=1}^n \left(d_i\sigma_i^2\right)} \left(1-\frac{d_{i_\mathrm{max}}}{m}\left(1-\phi_{i_\mathrm{max}}^2\right) \right)^{t},
$$
which means that increasing $\phi_j$ for $j\not\in M$ will not substantially influence convergence rate. 

Note that as soon as we have
\begin{equation}
\rho>
 1-\frac{d_{i}}{m}\left(1-\phi_{i}^2\right)  \label{eq: treshold}
\end{equation} 
for all $i$, the rate from theorem \ref{T: ng general convergence} will be driven by $\rho^k$ (as $k \rightarrow \infty$) and we will have

\begin{equation}
\E{ D(y^*)- D(y^{k}) } = \tilde{O}\left(\rho^k\right)
\label{eq: tildeO}
\end{equation}

One can think of the above as a threshold: if there is $i$ such that $\phi_i$ is large enough so that the inequality \eqref{eq: treshold} does not hold, the convergence rate is driven by $\phi_{i_\mathrm{max}}$. Otherwise, the rate  is not influenced by the insertion of noise. Thus, in theory, we do not pay anything in terms of performance as long as we do not hit the threshold. One might be interested in choosing $\phi_i$ so that the threshold is attained for all $i$, and thus $M=\{1,\dots,n\}$. This motivates the following result:

\begin{cor}
\label{corolary}
Let us choose 
\begin{equation}
\phi_i \eqdef \sqrt{1-\frac{\gamma}{d_i}}
\label{Eq: phi_i def}
\end{equation}
for all $i$, where  $\gamma \leq d_{\mathrm{min}}$.
Then
\begin{eqnarray*}
\E{ D(y^*)- D(y^{k}) } & \leq &
\left( 1-\min\left( \frac{\ac(\cG)}{2m},\frac{\gamma}{m}\right) \right)^k \left( D(y^*)- D(y^{0})+\frac{\sum_{i=1}^n \left(d_i\sigma_i^2\right)}{4m}k \right). 
\end{eqnarray*}

As a consequence, $\phi_i=\sqrt{1-\tfrac{\ac(\cG)}{2d_i}}$ is the largest decrease rate of noise for node $i$ such that the guaranteed convergence rate of the algorithm is not violated.
 
\label{C: noisy gossip special}
\end{cor}

While the above result clearly states the important threshold, it is not always practical as $\ac(\cG)$ might not be known. However, note that if we choose $\tfrac{nd_{\mathrm{min}}}{2(n-1)}\leq \gamma \leq d_{\mathrm{min}}$, we have 
$\min\left( \frac{\ac(\cG)}{2m},\frac{\gamma}{m}\right) = \frac{\ac(\cG)}{2m}$
since 
$\frac{\ac(\cG)}{2}\leq \frac{n}{n-1}\frac{d_{\mathrm{min}}}{2}\leq \gamma,$
where $e(\cG)$ denotes graph {\em edge connectivity}: the minimal number of edges to be removed so that the graph becomes disconnected.   Inequality $\ac(\cG) \leq \tfrac{n}{n-1}d_{\mathrm{min}}$ is a well known result in spectral graph theory \cite{fiedler1973algebraic}. As a consequence, if for all $i$ we have
$$
\phi_i\leq \sqrt{1-\frac{(n-1)d_\mathrm{min}}{2nd_i}},
$$
then the convergence rate is not driven by the noise.

\section{Numerical Evaluation}
\label{sec:experiments}

We devote this section to experimentally evaluate the performance of the Algorithms~\algB, \algE\ and \algN\ we proposed in the previous sections, applied to the Average Consensus problem. In the following experiments, we used the following popular graph topologies.
\begin{itemize}
\item {\em Cycle graph} with $n$ nodes: $\mathcal{C}(n)$. In our experiments we choose $n=10$. This small simple graph with regular topology is chosen for illustration purposes.
\item {\em Random geometric graph} with $n$ nodes and radius $r$:  $\mathcal{G}(n,r)$. Random geometric graphs \cite{penrose2003random} are very important in practice because of their particular formulation which is ideal for modeling wireless sensor networks \cite{gupta2000capacity, boyd2006randomized}. In our experiments we focus on a $2$-dimensional randomized geometric graph $\cG(n,r)$ which is formed by placing $n$ nodes uniformly at random in a unit square with edges between nodes which are having euclidean distance less than the given radius $r$. We set this to be to be $r = r(n) = \sqrt{\log(n)/n}$ --- it is well know that the connectivity is preserved in this case~\cite{gupta2000capacity}. We set $n = 100$.
\end{itemize}

An illustration of the two graphs appears is in Figure~\ref{Illustration}. 

\begin{figure}[H]
\centering
\begin{subfigure}{.45\textwidth}
  \centering
  \includegraphics[width=1\linewidth]{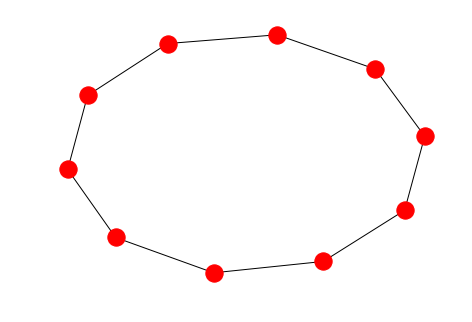}
  \caption{Cycle Graph: $\cC(10)$ }
\end{subfigure}%
\begin{subfigure}{.45\textwidth}
  \centering
  \includegraphics[width=1\linewidth]{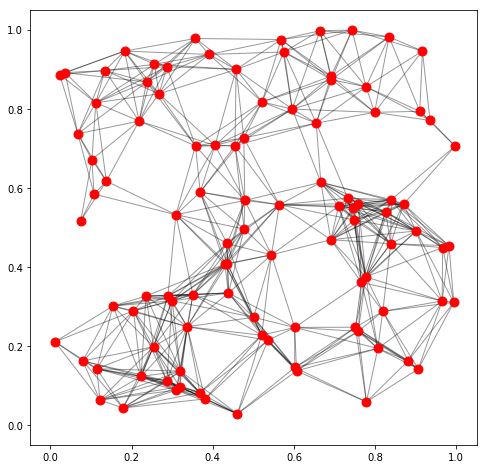}
  \caption{Random Geometric Graph: $\cG(n,r)$ }
\end{subfigure}
\caption{{Illustration of the two graph topologies we focus on in this section.}}
\label{Illustration}
\end{figure}

\textbf{Setup:}
In all experiments we generate a vector with of initial values $c_i$ from a uniform distribution over $[0,1]$. We run several experiments and present two kinds of figures that helps us to understand how the algorithms evolve and verify the theoretical results of the previous sections. These figures are:
\begin{enumerate}
\item The evolution of the initial values of the nodes. In these figures we plot how the trajectory of the values $x_i^k$ of each node $i$ evolves throughout iterations. The black doted horizontal line represents the exact average consensus value which all nodes should approach, and thus all other lines should approach this level.
\item The evolution of the relative error measure $q^t \eqdef \frac{\|x^t-x^*\|^2}{\|x^0 - x^*\|^2}.$
\end{enumerate}
We run each method for several parameters and for a pre-specified number of iterations not necessarily the same for each experiment. In each figure we have the relative error, both in normal scale or logarithmic scale, on the vertical axis and number of iterations on the horizontal axis.

To illustrate the first concept, we provide a simple example with the evolution of the initial values $x_i^k$ for the case of the Standard Gossip algorithm \cite{boyd2006randomized} in Figure~\ref{ExactMethod}. The horizontal black dotted line represents the average consensus value. It is the exact average of the initial values $c_i$ of the nodes in the network.

\begin{figure}[H]
\centering
\begin{subfigure}{.45\textwidth}
  \centering
  \includegraphics[width=1\linewidth]{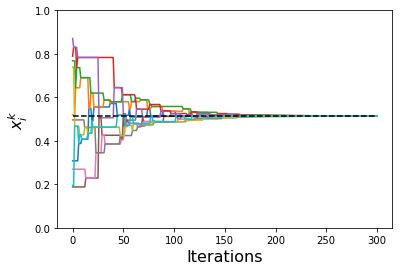}
  \caption{Cycle Graph}
\end{subfigure}%
\begin{subfigure}{.45\textwidth}
  \centering
  \includegraphics[width=1\linewidth]{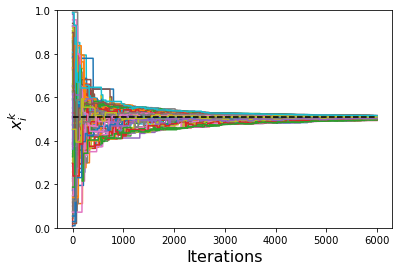}
  \caption{Random Geometric Graph}
\end{subfigure}\\
\caption{Trajectories of the values $x_i^t$ for the Standard Gossip algorithm for Cycle Graph and a random geometric graph. Each line corresponds to $x_i$ for some $i$.}
\label{ExactMethod}
\end{figure}

In the rest of this section we evaluate the performance of the novel algorithms we propose, and contrast with the above Standard Gossip algorithm, which we refer to as ``Baseline'' in the following figures labels.

\subsection{Private Gossip via Binary Oracle}

In this section we evaluate the performance of Algorithm~\algB\ presented in Section~\ref{sec:B}. In the algorithm, the input parameters are the positive stepsizes $\{\lambda^t\}_{t=0}^\infty$. The goal of the experiments is to compare the performance of the proposed algorithm using different choices of $\lambda^t$.

In particular, we use decreasing sequences of stepsizes $\lambda^t=1/t$ and $\lambda^t=1/ \sqrt{t}$, and three different fixed values for the stepsizes $\lambda^t=\lambda \in \{0.001, 0.01, 0.1\}$. We also include the adaptive choice $\lambda^t = \frac{1}{4m}\sum_{e\in \cE}|x_i^t-x_j^t|$ which we have proven to converge with linear rate in Theorem~\ref{thm: stepsize_adaptive}. We compare these choices in Figures \ref{Cycle10LamError} and \ref{RGG100LamErr}, along with the Standard Gossip algorithm for clear comparison.

\begin{figure}[H]
\centering
\begin{subfigure}{.3\textwidth}
  \centering
  \includegraphics[width=1\linewidth]{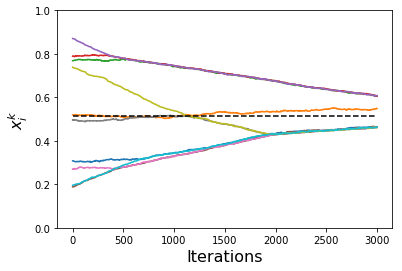}
  \caption{ $\lambda^t=\lambda=0.001$}
\end{subfigure}%
\begin{subfigure}{.3\textwidth}
  \centering
  \includegraphics[width=1\linewidth]{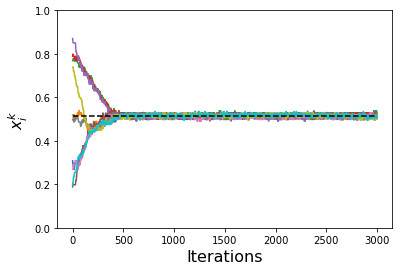}
  \caption{$\lambda^t=\lambda=0.01$}
\end{subfigure}
\begin{subfigure}{.3\textwidth}
  \centering
  \includegraphics[width=1\linewidth]{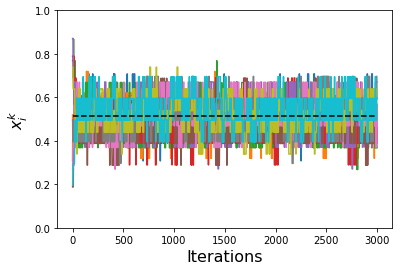}
  \caption{$\lambda^t=\lambda=0.1$}
\end{subfigure}\\
\begin{subfigure}{.3\textwidth}
  \centering
  \includegraphics[width=1\linewidth]{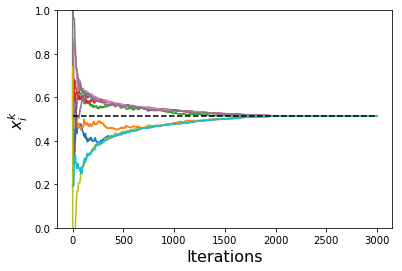}
  \caption{$\lambda^t=\frac{1}{k}$}
\end{subfigure}
\begin{subfigure}{.3\textwidth}
  \centering
  \includegraphics[width=1\linewidth]{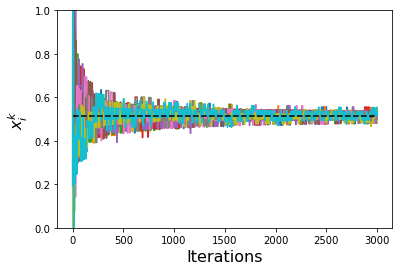}
  \caption{ $\lambda^t=\frac{1}{\sqrt{k}}$}
\end{subfigure}
\begin{subfigure}{.3\textwidth}
  \centering
  \includegraphics[width=1\linewidth]{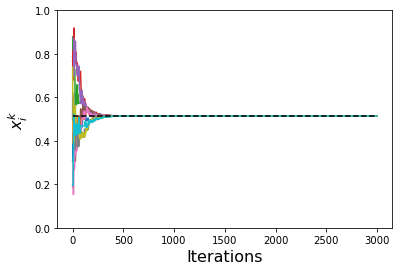}
  \caption{ $\lambda^t=\text{Adaptive}$}
\end{subfigure}
\caption{Trajectories of the values of $x_i^t$ for Binary Oracle run on the cycle graph.}
\label{Cycle10Lam}
\end{figure}

\begin{figure}[H]
\centering
\begin{subfigure}{.45\textwidth}
  \centering
  \includegraphics[width=1\linewidth]{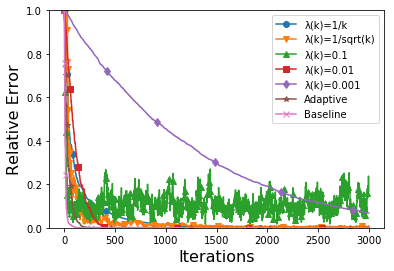}
  \caption{Linear Scale}
\end{subfigure}%
\begin{subfigure}{.45\textwidth}
  \centering
  \includegraphics[width=1\linewidth]{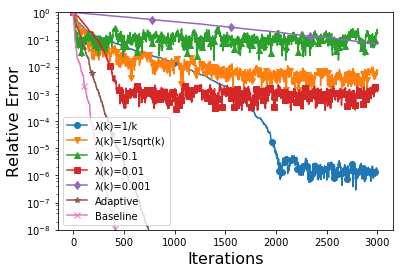}
  \caption{Logarithmic Scale}
\end{subfigure}
\caption{Convergence of the Binary Oracle run on the cycle graph.}
\label{Cycle10LamError}
\end{figure}
 
In general, we clearly see what is expected with the constant stepsizes --- that they converge to a certain neighbourhood and oscillate around optimum. With smaller stepsize, this neighbourhood is more accurate, but it takes longer to reach. With decreasing stepsizes, Theorem~\ref{thm:jhs988sh} suggests that $\lambda^t$ of order $1 / \sqrt{t}$ should be optimal. Figure~\ref{RGG100LamErr} demonstrates this, as the choice of $\lambda^t = 1/t$ decreases the stepsizes too quickly. However, this is not the case in Figure~\ref{Cycle10LamError} in which we observe the opposite effect. This is due to the cycle graph being small and simple, and hence the diminishing stepsize becomes problem only after relatively large number of iterations. With the adaptive choice of stepsizes, we recover linear convergence rate as predicted by Theorem~\ref{thm: stepsize_adaptive}.

The results in Figure~\ref{RGG100LamErr} show one surprising comparison. The adaptive choice of stepsizes does not seem to perform better than $\lambda^t = 1 / \sqrt{t}$. However, we verified that when running for more iterations, the linear rate of adaptive stepsize is present and converges significantly faster to higher accuracies. We chose to present the results for $6000$ iterations since we found it overall more clean.

\begin{figure}[H]
\centering
\begin{subfigure}{.3\textwidth}
  \centering
  \includegraphics[width=1\linewidth]{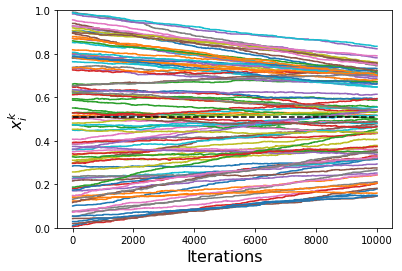}
  \caption{ $\lambda^t=\lambda=0.001$}
\end{subfigure}%
\begin{subfigure}{.3\textwidth}
  \centering
  \includegraphics[width=1\linewidth]{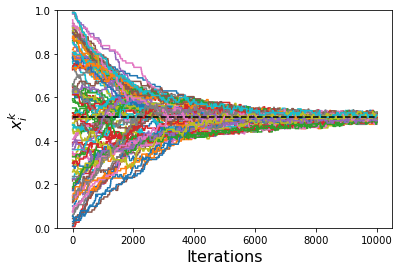}
  \caption{$\lambda^t=\lambda=0.01$}
\end{subfigure}
\begin{subfigure}{.3\textwidth}
  \centering
  \includegraphics[width=1\linewidth]{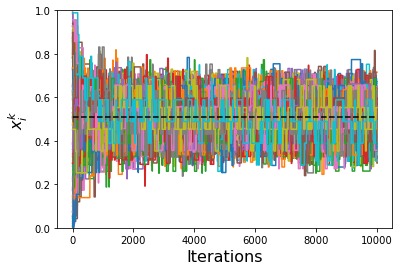}
  \caption{ $\lambda^t=\lambda=0.1$}
\end{subfigure}\\
\begin{subfigure}{.3\textwidth}
  \centering
  \includegraphics[width=1\linewidth]{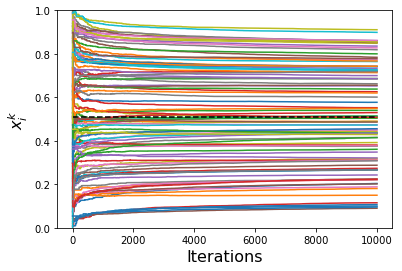}
  \caption{ $\lambda^t=\frac{1}{t}$}
\end{subfigure}
\begin{subfigure}{.3\textwidth}
  \centering
  \includegraphics[width=1\linewidth]{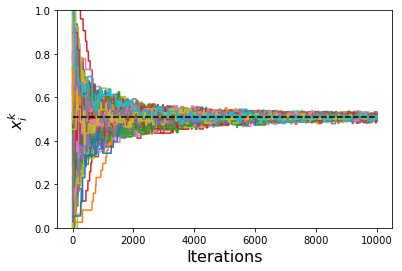}
  \caption{ $\lambda^t=\frac{1}{\sqrt{t}}$}
\end{subfigure}
\begin{subfigure}{.3\textwidth}
  \centering
  \includegraphics[width=1\linewidth]{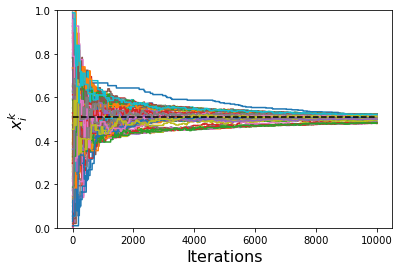}
  \caption{ $\lambda^t=\text{Adaptive}$}
\end{subfigure}
\caption{Trajectories of the values of $x_i^t$ for Binary Oracle run on the random geometric graph.}
\label{RGG100Lam}
\end{figure}

\begin{figure}[H]
\centering
\begin{subfigure}{.45\textwidth}
  \centering
  \includegraphics[width=1\linewidth]{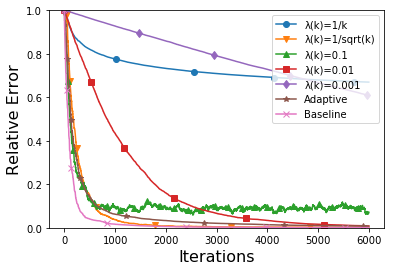}
  \caption{Linear Scale}
\end{subfigure}%
\begin{subfigure}{.45\textwidth}
  \centering
  \includegraphics[width=1\linewidth]{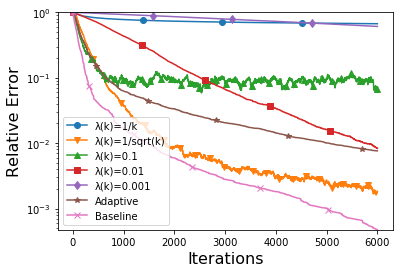}
  \caption{Logarithmic Scale}
\end{subfigure}
\caption{Convergence of the Binary Oracle run on the random geometric graph.}
\label{RGG100LamErr}
\end{figure}

\subsection{Private Gossip via $\epsilon$-Gap Oracle}

In this section we evaluate the performance of the Algorithm~\algE\ presented in Section~\ref{sec:E}. In the algorithm, the input parameter is the positive error tolerance variable $\epsilon$. For experimental evaluation. we choose three different values for the input, $\epsilon \in \{0.2, 0.02, 0.002\}$, and again use the same cycle and random geometric graphs. The trajectories of the values $x_i^t$ are presented in Figures~\ref{Cycle10Gap} and \ref{RGG100Gap}, respectively. The performance of the algorithm in terms of the relative error is presented in Figures~\ref{Cycle10GapErr} and \ref{RGG100GapErr}.

The performance is exactly matching the expectation --- with larger $\epsilon$, the method converges very fast to a wide neighbourhood of the optimum. For a small value, it converges much closer to the optimum, but it requires more iterations.

\begin{figure}[H]
\centering
\begin{subfigure}{.3\textwidth}
  \centering
  \includegraphics[width=1\linewidth]{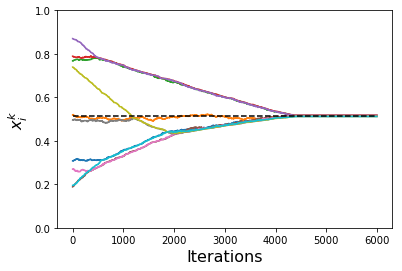}
  \caption{$\epsilon=0.002$}
\end{subfigure}
\begin{subfigure}{.3\textwidth}
  \centering
  \includegraphics[width=1\linewidth]{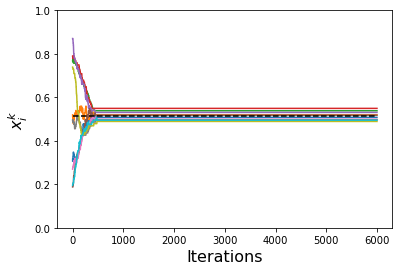}
  \caption{ $\epsilon=0.02$}
\end{subfigure}
\begin{subfigure}{.3\textwidth}
  \centering
  \includegraphics[width=1\linewidth]{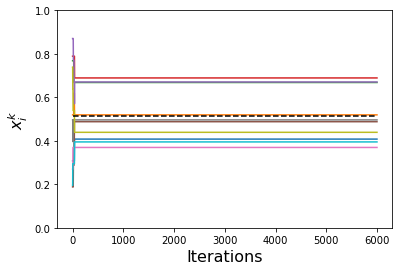}
  \caption{$\epsilon=0.2$}
\end{subfigure}
\caption{Trajectories of the values of $x_i^t$ for $\epsilon$-Gap Oracle run on the cycle graph.}
\label{Cycle10Gap}
\end{figure}

\begin{figure}[H]
\centering
\begin{subfigure}{.45\textwidth}
  \centering
  \includegraphics[width=1\linewidth]{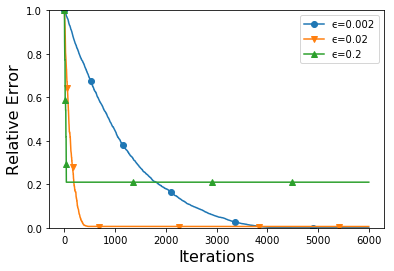}
  \caption{Linear Scale}
\end{subfigure}%
\begin{subfigure}{.45\textwidth}
  \centering
  \includegraphics[width=1\linewidth]{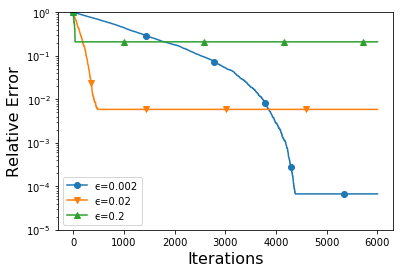}
  \caption{Logarithmic Scale}
\end{subfigure}
\caption{Convergence of the $\epsilon$-Gap Oracle run on the cycle graph.}
\label{Cycle10GapErr}
\end{figure}

\begin{figure}[H]
\centering
\begin{subfigure}{.3\textwidth}
  \centering
  \includegraphics[width=1\linewidth]{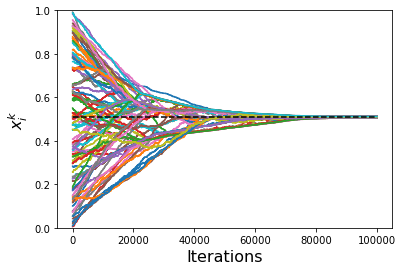}
  \caption{$\epsilon=0.002$}
\end{subfigure}
\begin{subfigure}{.3\textwidth}
  \centering
  \includegraphics[width=1\linewidth]{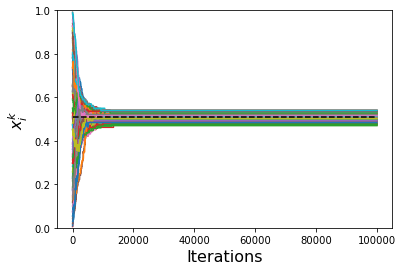}
  \caption{$\epsilon=0.02$}
\end{subfigure}
\begin{subfigure}{.3\textwidth}
  \centering
  \includegraphics[width=1\linewidth]{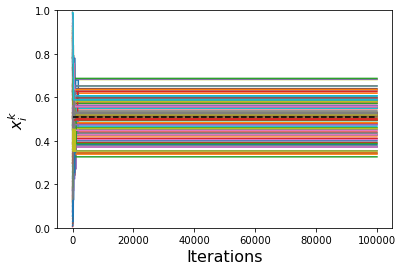}
  \caption{$\epsilon=0.2$}
\end{subfigure}
\caption{Trajectories of the values of $x_i^t$ for $\epsilon$-Gap Oracle run on the random geometric graph.}
\label{RGG100Gap}
\end{figure}

\begin{figure}[H]
\centering
\begin{subfigure}{.45\textwidth}
  \centering
  \includegraphics[width=1\linewidth]{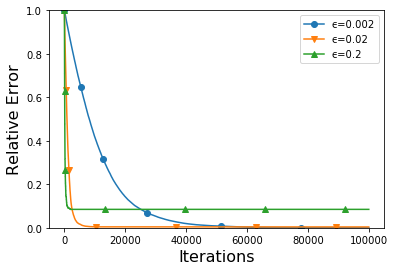}
  \caption{Linear Scale}
\end{subfigure}%
\begin{subfigure}{.45\textwidth}
  \centering
  \includegraphics[width=1\linewidth]{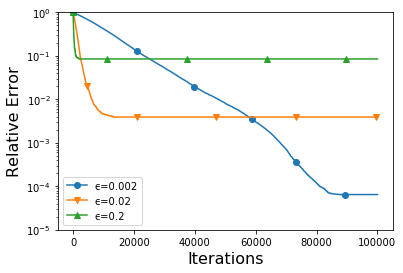}
  \caption{Logarithmic Scale}
\end{subfigure}
\caption{Convergence of the $\epsilon$-Gap Oracle run on the random geometric graph.}
\label{RGG100GapErr}
\end{figure}

\subsection{Private Gossip via Controlled Noise Insertion}

In this section we evaluate the performance of Algorithm~\algN\ presented in Section~\ref{sec: noise}. This algorithm has two different parameters for each node $i$. These are the initial variance $\sigma_i^2 \geq 0$ and the rate of decay, $\phi_i$,  of the noise.

To evaluate the impact of these parameters, we perform several experiments. As earlier, we use the same graph structures for evaluation: cycle graph and random geometric graph. The algorithm converges with a linear rate depending on maximum of two factors --- see Theorem~\ref{T: ng general convergence} and Corollary~\ref{C: noisy gossip special}. We will verify that this is indeed the case, and for values of $\phi_i$ above a certain threshold, the convergence is driven by the rate at which the noise decays. This is true for both identical values of $\phi_i$ for all $i$, and for varying values as per \eqref{Eq: phi_i def}. We further demonstrate the latter is superior in the sense that it enables insertion of more noise, without sacrificing the convergence speed. Finally, we study the effect of various magnitudes of the noise inserted initially.

\subsubsection{Fixed variance, identical decay rates}

In this part, we run Algorithm~\algN\ with $\sigma_i = 1$ for all $i$, and set $\phi_i = \phi$ for all $i$ and some $\phi$. We study the effect of varying the value of $\phi$ on the convergence of the algorithm. 

In both Figures~\ref{Cycle10NoiseErr}b and \ref{RGG100NoiseErr}b, we see that for small values of $\phi$, we eventually recover the same rate of linear convergence as the Standard Gossip algorithm. If the value of $\phi$ is sufficiently close to $1$ however, the rate is driven by the noise and not by the convergence of the Standard Gossip algorithm. This value is $\phi = 0.98$ for cycle graph, and $\phi=0.995$ for the random geometric graph in the plots we present.

Looking at the individual runs for small values of $\phi$ in Figure~\ref{RGG100NoiseErr}b, we see some variance in terms of when the asymptotic rate is realized. We would like to point out that this \emph{does not} provide additional insight into whether specific small values of $\phi$ are in general better for the following reason. The Standard Gossip algorithm is itself a randomized algorithm, with an inherent uncertainty in the convergence of any particular run. If we ran the algorithms multiple times, we observe variance in the evolution of the suboptimality of similar magnitude, just as what we see in the figure. Hence, the variance is expected, and not significantly influenced by the noise.

\begin{figure}[H]
\centering
\begin{subfigure}{.3\textwidth}
  \centering
  \includegraphics[width=1\linewidth]{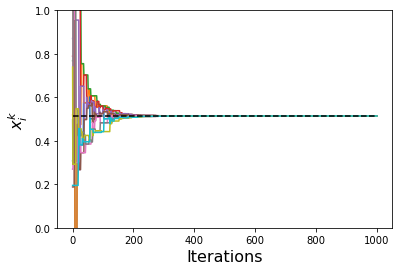}
  \caption{$\phi=0.001$}
\end{subfigure}
\begin{subfigure}{.3\textwidth}
  \centering
  \includegraphics[width=1\linewidth]{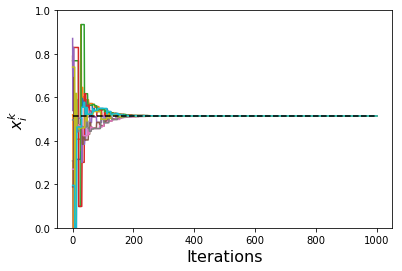}
  \caption{$\phi=0.01$}
\end{subfigure}
\begin{subfigure}{.3\textwidth}
  \centering
  \includegraphics[width=1\linewidth]{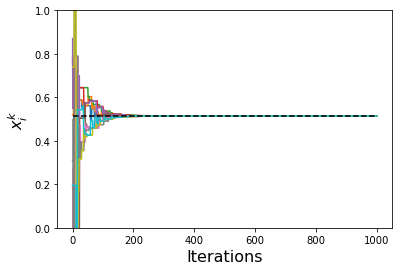}
  \caption{ $\phi=0.1$}
\end{subfigure}\\
\begin{subfigure}{.3\textwidth}
  \centering
  \includegraphics[width=1\linewidth]{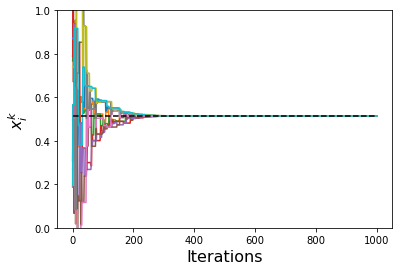}
  \caption{ $\phi=0.5$}
\end{subfigure}
\begin{subfigure}{.3\textwidth}
  \centering
  \includegraphics[width=1\linewidth]{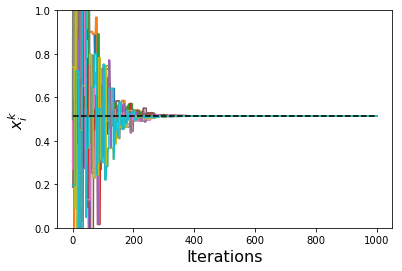}
  \caption{ $\phi=0.9$}
\end{subfigure}
\begin{subfigure}{.3\textwidth}
  \centering
  \includegraphics[width=1\linewidth]{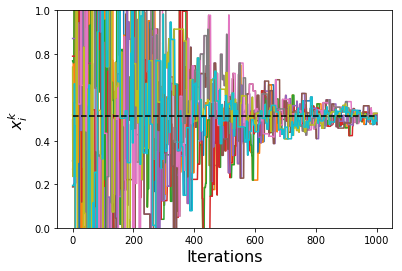}
  \caption{ $\phi=0.98$}
\end{subfigure}
\caption{Trajectories of the values of $x_i^t$ for Controlled Noise Insertion run on the cycle graph for different values of $\phi$.}
\label{Cycle10Noise}
\end{figure}

\begin{figure}[H]
\centering
\begin{subfigure}{.45\textwidth}
  \centering
  \includegraphics[width=1\linewidth]{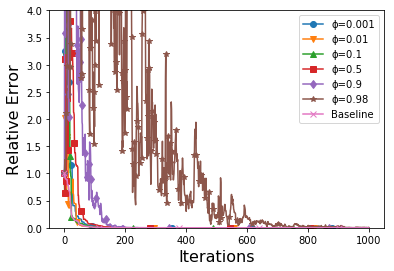}
  \caption{Linear Scale}
\end{subfigure}%
\begin{subfigure}{.45\textwidth}
  \centering
  \includegraphics[width=1\linewidth]{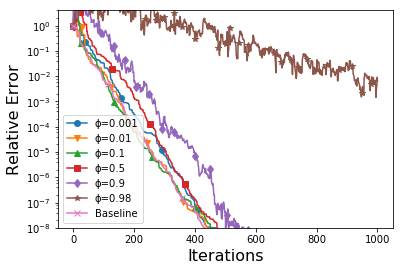}
  \caption{Logarithmic Scale}
\end{subfigure}
\caption{Convergence of the Controlled Noise Insertion run on the cycle graph for different values of $\phi$.}
\label{Cycle10NoiseErr}
\end{figure}

\begin{figure}[H]
\centering
\begin{subfigure}{.3\textwidth}
  \centering
  \includegraphics[width=1\linewidth]{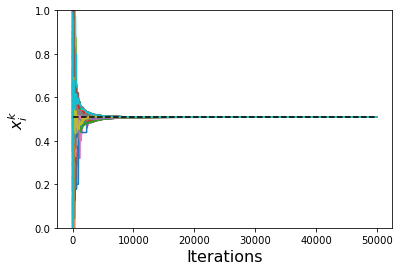}
  \caption{ $\phi=0.001$}
\end{subfigure}%
\begin{subfigure}{.3\textwidth}
  \centering
  \includegraphics[width=1\linewidth]{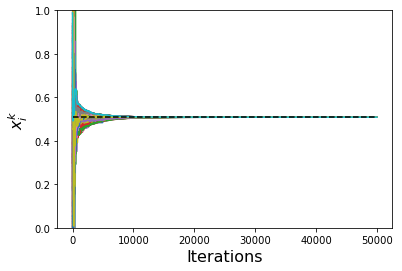}
  \caption{$\phi=0.01$}
\end{subfigure}
\begin{subfigure}{.3\textwidth}
  \centering
  \includegraphics[width=1\linewidth]{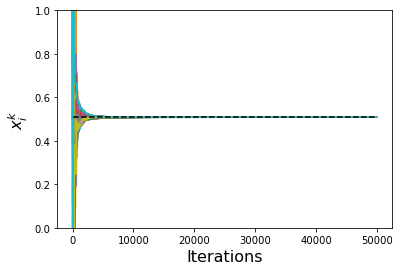}
  \caption{$\phi=0.1$}
\end{subfigure}\\
\begin{subfigure}{.3\textwidth}
  \centering
  \includegraphics[width=1\linewidth]{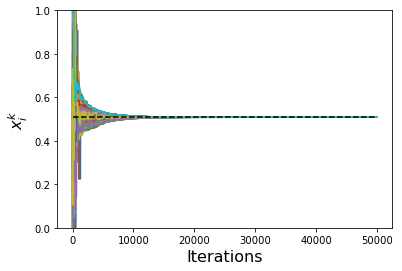}
  \caption{ $\phi=0.5$}
\end{subfigure}
\begin{subfigure}{.3\textwidth}
  \centering
  \includegraphics[width=1\linewidth]{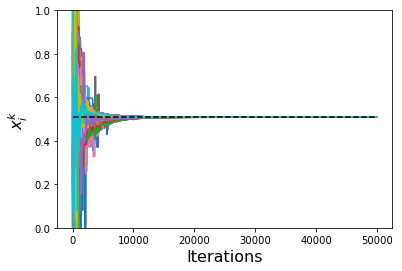}
  \caption{ $\phi=0.9$}
\end{subfigure}
\begin{subfigure}{.3\textwidth}
  \centering
  \includegraphics[width=1\linewidth]{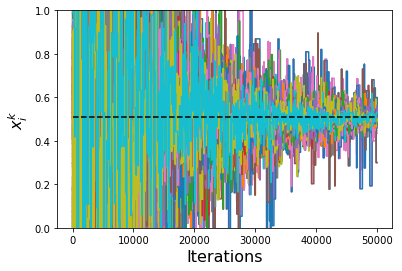}
  \caption{ $\phi=0.995$}
\end{subfigure}
\caption{Trajectories of the values of $x_i^t$ for Controlled Noise Insertion run on the random geometric graph for different values of $\phi$.}
\label{RGG100Noise}
\end{figure}

\begin{figure}[H]
\centering
\begin{subfigure}{.45\textwidth}
  \centering
  \includegraphics[width=1\linewidth]{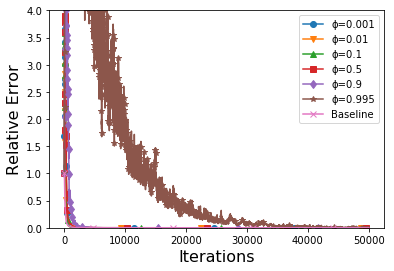}
  \caption{Linear Scale}
\end{subfigure}%
\begin{subfigure}{.45\textwidth}
  \centering
  \includegraphics[width=1\linewidth]{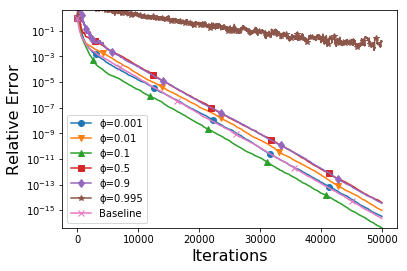}
  \caption{Logarithmic Scale}
\end{subfigure}
\caption{Convergence of the Controlled Noise Insertion run on the random geometric graph for different values of $\phi$.}
\label{RGG100NoiseErr}
\end{figure}

\subsubsection{Variance 1 and different decay rates}

In this section, we perform similar experiment as above, but let the values $\phi_i$ be vary for different nodes $i$. This is controlled by the choice of $\gamma$ as in \eqref{Eq: phi_i def}. Note that by decreasing $\gamma$, we increase $\phi_i$, and thus smaller $\gamma$ means the noise decays at a slower rate. Here, due to the regular structure of the cycle graph, we present only results for the random geometric graph.

It is not straightforward to compare this setting with the setting of identical $\phi_i$, and we return to it in the next section. Here we only remark that we again see the existence of a threshold predicted by theory, beyond which the convergence is dominated by the inserted noise. Otherwise, we recover the rate of the Standard Gossip algorithm.

\begin{figure}[H]
\centering
\begin{subfigure}{.3\textwidth}
  \centering
  \includegraphics[width=1\linewidth]{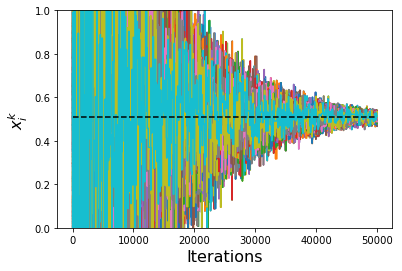}
  \caption{$\gamma=0.1$}
\end{subfigure}
\begin{subfigure}{.3\textwidth}
  \centering
  \includegraphics[width=1\linewidth]{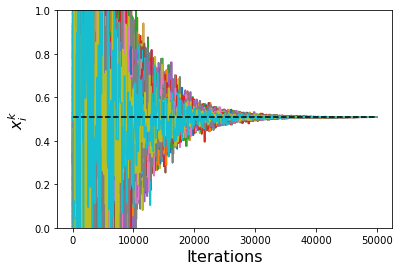}
  \caption{$\gamma=0.2$}
\end{subfigure}
\begin{subfigure}{.3\textwidth}
  \centering
  \includegraphics[width=1\linewidth]{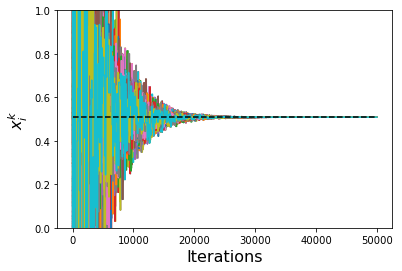}
  \caption{$\gamma=0.3$}
\end{subfigure}\\
\begin{subfigure}{.3\textwidth}
  \centering
  \includegraphics[width=1\linewidth]{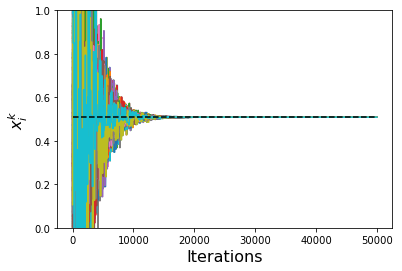}
  \caption{$\gamma=0.5$}
\end{subfigure}
\begin{subfigure}{.3\textwidth}
  \centering
  \includegraphics[width=1\linewidth]{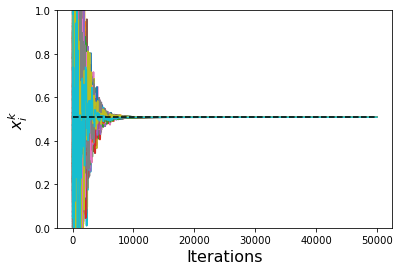}
  \caption{$\gamma=1$}
\end{subfigure}
\begin{subfigure}{.3\textwidth}
  \centering
  \includegraphics[width=1\linewidth]{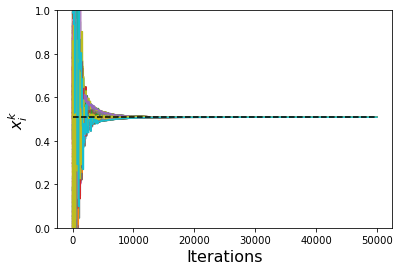}
  \caption{$\gamma=2$}
\end{subfigure}\\
\caption{Trajectories of the values of $x_i^t$ for Controlled Noise Insertion run on the random geometric graph for different values of $\phi_i$, controlled by $\gamma$.}
\label{Cycle10Gamma}
\end{figure}

\begin{figure}[H]
\centering
\begin{subfigure}{.45\textwidth}
  \centering
  \includegraphics[width=1\linewidth]{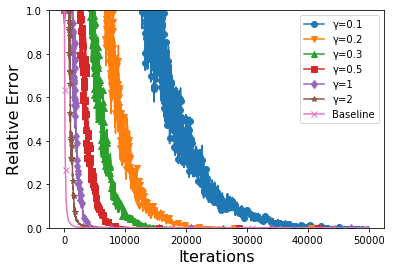}
  \caption{Normal Scale}
\end{subfigure}%
\begin{subfigure}{.45\textwidth}
  \centering
  \includegraphics[width=1\linewidth]{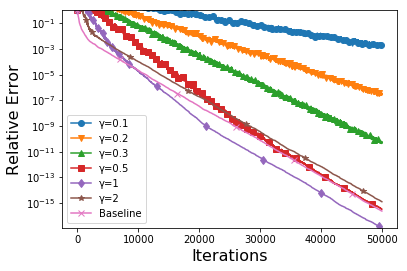}
  \caption{Logarithmic Scale}
\end{subfigure}
\caption{Convergence of the Controlled Noise Insertion run on the random geometric graph for different values of $\phi_i$, controlled by $\gamma$.}
\label{Cycle10GammaErr}
\end{figure}

\subsubsection{Impact of varying $\phi_i$}

In this experiment, we demonstrate the practical utility of letting the rate of decay $\phi_i$ to be different on each node $i$. In order to do so, we run experiment on the random geometric graph and compare the settings investigated in the previous two sections --- the noise decay rate driven by $\phi$, or by $\gamma$.

In first place, we choose the values of $\phi_i$ such that that the two factors in Corollary~\ref{C: noisy gossip special} are equal. For the particular graph we used, this corresponds to $\gamma \approx 0.17$ with $\phi_i=\sqrt{1-\frac{\ac(\cG)}{2d_i}}$. Second, we make the factors equal, but with constraint of having $\phi_i$ to be equal for all $i$. This corresponds to $\phi_i \approx 0.983$ for all $i$.

The performance for a large number of iterations is displayed in left side of Figure~\ref{fig:comparison_varying_phi}. We see that theabove two choices indeed yield very similar practical performance, which also eventually matches the rate predicted by theory. For complete comparison, we also include performance of the Standard Gossip algorithm.

The important message is conveyed in the histogram in the right side of Figure~\ref{fig:comparison_varying_phi}. The histogram shows the distribution of the values of $\phi_i$ for different nodes $i$. The minimum of these values is what we needed in the case of identical $\phi_i$ for all $i$. However, most of the values are significantly higher. This means, that if we allow the noise decay rates to depend on the number of neighbours, we are able to increase the amount of noise inserted, without sacrificing practical performance. This is beneficial, as more noise will likely be beneficial for any formal notion of protection of the initial values.

\begin{figure}[H]
\centering
\begin{subfigure}{.45\textwidth}
  \centering
  \includegraphics[width=1\linewidth]{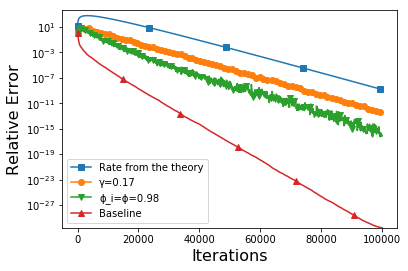}
\end{subfigure}%
\begin{subfigure}{.45\textwidth}
  \centering
  \includegraphics[width=1\linewidth]{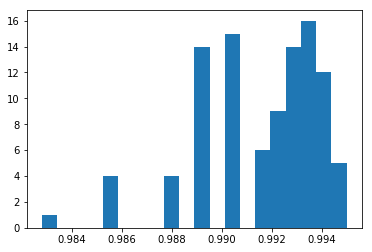}
\end{subfigure}

\caption{Left: Performance of the noise oracle with noise decrease rate chosen according to Corollary \ref{C: noisy gossip special}. Right: Histogram of of distribution of $\phi_i$}
\label{fig:comparison_varying_phi}
\end{figure}

\section{Conclusion}
\label{sec:conclusion}

In this work we addressed the Average Consensus problem via novel asynchronous randomized gossip algorithms. We propose algorithmic tools for protection of the private values each node in the network holds initially. However, we do not quantify any formal notion of privacy protection achievable using these tools; this is left for future research.

In particular, we propose two ways to achieve this goal. First, which we believe is the first of its kind, weakens the oracle used in the gossip framework, to provide only categorical (or even binary) information to each participating node about the value of the other node. In the second approach, we systematically inject and withdraw noise throughout the iterations, so as to ensure convergence to the average consensus value. In all cases, we provide explicit convergence rates and evaluate practical convergence on common simulated network topologies.

\bibliographystyle{plain}
\bibliography{biblio}

\newpage
\appendix

\section{Proofs for Section~\ref{sec:introduction}}
\subsection{Proof of Lemma \ref{L: rel measures}}
\begin{lem}

\begin{equation}
\sum_{i=1}^n \left( \sum_{j=1}^n (x_j-x_i) \right)^2
=
\frac{n}{2}\sum_{i=1}^n \sum_{j=1}^n(x_j-x_i)^2 
\label{Eq: sum lemma}
\end{equation}

\end{lem}
\begin{proof}
Using simple algebra we have
\begin{eqnarray*}
\sum_{i=1}^n \left( \sum_{j=1}^n (x_j-x_i) \right)^2
&=&
\sum_{i=1}^n \left( \sum_{j=1}^n x_j-nx_i \right)^2
\\
&=&
\sum_{i=1}^n \left( \left( \sum_{j=1}^n x_j \right)^2 +n^2x_i^2-2nx_i \left( \sum_{j=1}^n x_j \right) \right)
\\
&=&
n\left(\sum_{j=1}^n x_j \right)^2
+
n^2\sum_{i=1}^nx_i^2
-2n\left(  \sum_{j=1}^n x_j \right)^2
\\
&=&
n^2\sum_{i=1}^nx_i^2
-n\left( \sum_{i=1}^n x_i \right)^2 .
\end{eqnarray*}

Manipulating right hand side of \eqref{Eq: sum lemma} we obtain

\begin{eqnarray*}
\frac{n}{2}\sum_{i=1}^n \sum_{j=1}^n(x_j-x_i)^2 &=&
\frac{n}{2}\sum_{i=1}^n \sum_{j=1}^n\left(x_j^2+x_i^2-2x_ix_j\right)
\\
&=& 
n^2\sum_{i=1}^nx_i^2 -n\sum_{i=1}^n \sum_{j=1}^nx_ix_j
\quad =\quad 
n^2\sum_{i=1}^nx_i^2
-n\left( \sum_{i=1}^n x_i \right)^2 .
\end{eqnarray*}

Clearly, LHS and RHS of \eqref{Eq: sum lemma} are equal.

\end{proof}

In order to show \eqref{Eq: equality} it is enough to notice that

\begin{eqnarray*}
\|\bar{c}\ones - x\|^2 
&=&  \sum_{i=1}^n (\bar{c}-x_i)^2 
\quad = \quad 
 \sum_{i=1}^n \left(\frac{1}{n}\sum_{j=1}^nx_j  -x_i\right)^2 
 \\
&=&
 \sum_{i=1}^n \left(\sum_{j=1}^n \frac{1}{n} (x_j  -x_i)\right)^2 
 \quad
\overset{\eqref{Eq: sum lemma}}{=} \quad
 \frac{1}{2} \sum_{i=1}^n \sum_{j=1}^n \frac{1}{n} (x_j  -x_i)^2 
 \\
&=& \frac{1}{2n} \sum_{i=1}^n \sum_{j=1}^n  (x_j  -x_i)^2.
\end{eqnarray*}

Note that we have 
$$
\frac{1}{nm}\left(\sum_{e=(i,j)\in \cE}|x_i-x_j|\right)^2 \leq \frac{1}{n}\sum_{e\in \cE}(x_i-x_j)^2\leq \frac{1}{n}\sum_{(i,j)}(x_i-x_j)^2 \stackrel{\eqref{Eq: equality}}{=} \|\bar{c} \ones -x\|^2,
$$
which proves \eqref{Eq: s upper bound}.  On the other hand, we have
$$
\frac{1}{\ac(\cG)}  \left(\sum_{e=(i,j)\in \cE}|x_i-x_j|\right)^2
\geq
\frac{1}{\ac(\cG)} \sum_{e\in \cE} \left(x_i-x_j \right)^2
\stackrel{\eqref{eq:8998gd98gikd}}{\geq}
\| \bar{c} \ones -x \|^2,
$$
which concludes \eqref{Eq: s lower bound}. Inequality \eqref{Eq: delta bound} holds trivially.  

\section{Proofs for Section~\ref{sec:gossip}} \label{sec:App-Sec3}
 
We now perform the analysis of Algorithm~\algG.

\subsection{Proof of Lemma~\ref{lem:beta}}

\begin{lem}
The eigenvalues of $\tilde{\mL}=n\mI - \ones \ones^\top$ are $\{0,n,n,\dots,n\}$
\label{L: eigenvalue of tilde L}
\end{lem}
\begin{proof}
Clearly, $\tilde{\mL} \ones=0$. Consider some vector $x$ such that $\langle x, \ones\rangle=0$. Then, $\tilde{\mL}x= n\mI x-\ones\ones^{\top} x=nx+\ones\ones^{\top} x=nx$ thus $x$ is an eigenvector corresponding to eigenvalue $n$. Thus, we can pick $n-1$ linearly independent eigenvectors of $\tilde{\mL}$ corresponding to eigenvalue $n$,  which concludes the proof. 
\end{proof} 

The Laplacian matrix of $\cG$ is the matrix $\mL=\bA^\top \bA$. We have 
$\mL_{ii} = d_i$ (degree of vertex $i$), $\mL_{ij} = \bL_{ji} = -1$ if $(i,j)\in \cE$ and $\bL_{ij} = 0$ otherwise. A simple computation reveals that for any $x\in \R^n$ we have

\[x^\top \bL x = \sum_{e=(i,j)\in \cE} (x_i-x_j)^2.\]

Let $\tilde{\bA}$ be the $n(n-1)/2 \times n$  matrix corresponding to the complete graph $\tilde{\cG}$ on $\cV$. Let $\tilde{\bL} = \tilde{\bA}^\top \tilde{\bA}$ be its Laplacian. We have $\tilde{\bL}_{ii}=n-1$ for all $i$ and $\tilde{\bL}_{ij} =-1$ for $i\neq j$. So, $\tilde{\bL} = n \mI - \ones \ones^\top$.  Then

\[x^\top \tilde{\bL} x = n\|x\|^2 -\left(\sum_{i=1}^n x_i\right)^2=\sum_{(i,j)} (x_i - x_j)^2.\]

Inequality \eqref{eq:hdgugvej} can therefore be recast as follows:

\[x^\top (n \mI - \ones \ones^\top) x \leq  x^\top \beta(\cG) \bL x, \qquad x\in \R^n.\]

Let $\beta=\beta(\cG)$. Note that both $\tilde{\mL}$ and $\beta \mL$ are Hermitian thus have real eigenvalues and there exist an orthonormal basis of their eigenvectors. Suppose that $\{x_1,\dots x_n\}$ are eigenvectors of $\beta L$ corresponding to eigenvalues $\lambda_1(\beta \mL),\lambda_2(\beta \mL) \dots ,\lambda_n(\beta \mL)$. Without loss of generality assume that these eigenvectors form an orthonormal basis and $\lambda_1(\beta \mL)\geq \dots \geq \lambda_n(\beta \mL)$

Clearly, $\lambda_n(\beta \mL)=0$, $x_n=\ones/\sqrt{n}$, and $\lambda_{n-1}(\beta \mL)=n$. Lemma~\ref{L: eigenvalue of tilde L}
states that eigenvalues of $\tilde{\mL}$ are $\{0,n,n,\dots,n\}$.

 One can easily see that eigenvector corresponding to zero eigenvalue of $\tilde{\mL}$ is $x_n$. Note that eigenvectors $x_1,\dots,x_{n-1}$ generate an eigenspace corresponding to eigenvalue $n$ of $\tilde{\mL}$. 

Consider some $x=\sum_{i=1}^n c_ix_i$, $c_i \in \R$ for all $i$. Then we have

$$
x^{\top}\tilde{\mL}x=
\sum_{i=1}^n \lambda_i\left(\tilde{\mL}\right)c_i^2\leq
\sum_{i=1}^n \lambda_i\left(\beta \mL \right)c_i^2=
 x^{\top}\beta \mL x,
$$
which concludes the proof. 

\subsection{Two Lemmas}

We first establish two lemmas which will be needed to prove Theorem~\ref{thm:G}.

\begin{lem} \label{lem:09u9djh9ffs}
Assume that edge $e=(i,j)$ is selected in iteration $t$ of Algorithm~\algG. Then \begin{equation} \label{eq:89g9s8guff} D(y^{t+1}) - D(y^t) = \tfrac{1}{4}(x^t_i-x^t_j)^2.\end{equation}
\end{lem}
\begin{proof} We have $y^{t+1} = y^t + \lambda^t f_e$ where $\lambda^t$ is chosen so that $D(y^{t+1})-D(y^t)$ is maximized. Applying Lemma~\ref{lem:98y98yss}, we have
\[
D(y^{t+1}) - D(y^t) =\max_{\lambda} -\lambda (x^t_i-x^t_j) - \lambda^2\\
=  \tfrac{1}{4}(x_i^t - x_j^t)^2.
\]
\end{proof}

\begin{lem} \label{L: D bound by beta} Let  $x\in \R^n$ such that $\frac{1}{n}\sum_i x_i = \bar{c}$. Then \begin{equation}
\label{eq:8998gd98gikd}\frac{1}{2}\|\bar{c}\ones - x\|^2 \leq \frac{1}{2\ac(\cG)} \sum_{e=(i,j)\in \cE} (x_i-x_j)^2.
\end{equation}
\end{lem}

\begin{proof}
\begin{eqnarray*}\frac{1}{2}\|\bar{c}\ones - x\|^2 
&\overset{\eqref{Eq: equality}}{=}& \frac{1}{4n} \sum_{i=1}^n \sum_{j=1}^n  (x_j  -x_i)^2
\quad = \quad \frac{1}{2n} \sum_{(i,j)} (x_j  -x_i)^2\\
&\overset{\eqref{eq:hdgugvej}}{\leq}& \frac{\beta(\cG)}{2n} \sum_{e=(i,j)\in \cE} (x_i-x_j)^2
\quad \overset{\text{Lemma}~\ref{lem:beta}}{=} \quad
\frac{1}{2\ac(\cG)} \sum_{e=(i,j)\in \cE} (x_i-x_j)^2
\end{eqnarray*}

\end{proof}

\subsection{Proof of Theorem~\ref{thm:G}}

Having established Lemmas \ref{lem:09u9djh9ffs} and \ref{L: D bound by beta}, we can now proceed with the proof of Theorem~\ref{thm:G}:

\begin{eqnarray*}\E{D(y^*) - D(y^{t+1}) \;|\; y^t} &=& D(y^*) - D(y^{t}) - \E{D(y^{t+1}) - D(y^{t}) \;|\; y^t} 
\\
&\overset{\eqref{eq:89g9s8guff}}{=}&  D(y^*) - D(y^{t}) - \sum_{e=(i,j)\in \cE}\frac{1}{4m} (x^t_i-x^t_j)^2\\
&\overset{\eqref{eq:99d8gds}}{=}& \frac{1}{2}\|\bar{c}\ones - x^t\|^2- \sum_{e=(i,j)\in \cE}\frac{1}{4m} (x^t_i-x^t_j)^2\\
&\overset{\eqref{eq:8998gd98gikd}}{\leq} &  \left(1-\frac{\ac(\cG)}{2m }\right)\frac{1}{2}\|\bar{c}\ones - x^t\|^2\\
&\overset{\eqref{eq:99d8gds}}{=}& \left(1-\frac{\ac(\cG)}{2m}\right)\left(D(y^*) - D(y^{t}) \right) \ .
\end{eqnarray*}

Taking expectation again, we get the recursion
$$\E{D(y^*) - D(y^{t+1}) } \quad \leq \quad \left(1-\frac{\ac(\cG)}{2m}\right)\E{D(y^*) - D(y^{t}) }.$$

\section{Proofs for Section~\ref{sec:Private}}

\subsection{Proof of Theorem~\ref{thm:jhs988sh}}

\begin{lem} \label{L: optimal stepsizes binary}
 Fix $k\geq 0$ and let $R >0$. Then
\[\min_{\lambda = (\lambda^0,\dots,\lambda^k)\in \R^{k+1}} \frac{R + \beta^k}{\alpha^k} = 2\sqrt{\frac{R}{k+1}},\]
and the optimal solution is given by $\lambda^t = \sqrt{\tfrac{R}{k+1}}$ for all $t$.
\end{lem}
\begin{proof} Define $\phi(\lambda) = \tfrac{R+ \beta^k}{\alpha^k}$. If we write $\lambda = r x$, where $r=\|\lambda\|$ and $x$ is of unit norm, then $\phi(tx)=\tfrac{R+r^2}{r \langle \ones, x\rangle}$. Clearly, for any fixed $r$, the $x\in \R^{k+1}$  minimizing $x\mapsto \phi(rx)$ is $x=\ones/\|\ones\|$, where $\ones$ is the vector of ones in $\R^{k+1}$. It now only remains to minimize the function $r\mapsto \tfrac{R+r^2}{r \|\ones\|}$. This function is convex and differentiable. Setting the derivative to zero leads to $r = \sqrt{R}$. Combining the above, we get the optimal solution $\lambda = \frac{r }{\| \ones\|} \ones = \tfrac{\sqrt{R}}{\|  \ones \|} \ones$.
\end{proof}

Let $e=(i,j)$ be the edge selected at iteration $t\geq 0$. Applying Lemma~\ref{lem:98y98yss}, we see that
$
 D(y^{t+1}) - D(y^t)
= \lambda^t |x^t_i-x^t_j| - \left(\lambda^t\right)^2.
$ Taking expectation with respect to edge selection, we get 
\[\E{D(y^{t+1})-D(y^t) \;|\; y^t} = -\left(\lambda^t\right)^2 +  \lambda^t \cdot \frac{1}{m}\sum_{e=(i,j)\in \cE} |x^t_i - x^t_j|, \]
and taking expectation again and using the tower property, we get the identity
$\E{D(y^{t+1})-D(y^t) } = -\left(\lambda^t\right)^2 +  \lambda^t\cdot \E{L^t}.$
Therefore,
\begin{eqnarray*}
D(y^*) - D(y^0)& \geq & \E{D(y^{k+1})-D(y^0)} \\
&=& \E{\sum_{t=0}^{k}D(y^{t+1}) - D(y^t)} \\
&=& \sum_{t=0}^{k}\E{D(y^{t+1}) - D(y^t)} \quad = \quad -\sum_{t=0}^{k}\left(\lambda^t\right)^2 + \sum_{t=0}^{k}  \lambda^t \cdot \E{L^t}. \end{eqnarray*}
It remains to reshuffle the resulting inequality to obtain \eqref{Eq: binary general rate}.

We can see that part (i) follows directly. Optimality of stepsizes in (ii) is due to Lemma~\ref{L: optimal stepsizes binary}. To show (iii) we should state that

$$
\alpha^k=\sum_{t=0}^k \lambda^k=\sum_{1}^{k+1}\frac{a}{\sqrt{t}}\geq a \int_{t=1}^{k+2}t^{-1/2} dt=2a\left( \sqrt{k+2} -1\right)
$$

$$
\beta^k=\sum_{t=0}^k \left(\lambda^k\right)^2=\sum_{t=1}^{k+1}\frac{a^2}{t}\leq a^2 \int_{1/2}^{k+3/2}t^{-1} dt=a^2 \left(\log(k+3/2)+\log(2) \right) 
$$
The inequality above holds due to the fact that for $t>1/2$ we have $t^{-1}\leq \int_{t-1/2}^{t+1/2}x^{-1}dx $ since $x^{-1}$ is convex function. 

\subsection{Proof of Theorem \ref{thm: stepsize_adaptive}}
Using Lemma~\ref{lem:98y98yss} with we have
\begin{eqnarray*}
\E{D(y^{t+1})-D(y^t)\;|\; y^t}&=&-\left(\lambda^t\right)^2+\lambda^t \frac{1}{m}\sum_{e\in \cE}|x_i^t-x_j^t|
\\
&=&
\frac{1}{4m^2}\left(\sum_{e\in \cE}|x_i^t-x_j^t|\right)^2
\quad \geq \quad \frac{1}{4m^2}\sum_{e\in \cE}\left(x_i^t-x_j^t\right)^2.
\end{eqnarray*}
Taking the expectation again we obtain
\begin{equation} \label{eq: adaptive binary first}
\E{D(y^{t+1})-D(y^t)}
 \geq \frac{1}{4m^2}\E{\sum_{e\in \cE}\left(x_i^t-x_j^t\right)^2}.
\end{equation}

On the other hand, we have
\begin{eqnarray*}
D(y^{t+1})-D(y^t) &=& \left (D(y^{t+1})-D(y^*) \right) +\left(D(y^*)-D(y^t) \right)\\
&=&
\frac12 \|\overline{c}\ones-x^{t} \|^2
-
\frac12 \|\overline{c}\ones-x^{t+1} \|^2 
\\
&=&
\frac{\ac(\cG)}{4m^2}\|\overline{c}\ones-x^{t} \|^2+
\left(1-\frac{\ac(\cG)}{2m^2}\right)\frac12 \|\overline{c}\ones-x^{t} \|^2
-
\frac12 \|\overline{c}\ones-x^{t+1} \|^2 
\\
&\stackrel{\eqref{eq:8998gd98gikd}}{\leq}&
\frac{1}{4m^2}\sum_{e=(i,j)\in \cE}\left(x_i^t-x_j^t\right)^2+
\left(1-\frac{\ac(\cG)}{2m^2}\right)\frac12 \|\overline{c}\ones-x^{t} \|^2
-
\frac12 \|\overline{c}\ones -x^{t+1} \|^2 .
\end{eqnarray*}

Taking the expectation of the above and combining with \eqref{eq: adaptive binary first} we obtain the desired recursion
$$
\E{\|\overline{c}\ones-x^{t+1} \|^2} 
\quad \leq \quad
\left(1-\frac{\ac(\cG)}{2m^2}\right)
\E{\|\overline{c}\ones-x^{t} \|^2}
.$$

\subsection{Proof of Lemma~\ref{lem:09ys09y9ss}}

Let $e=(i,j)$ be the edge selected at iteration $t$. Applying Lemma~\ref{lem:98y98yss}, we see that

\[
 D(y^{t+1}) - D(y^t)
=\begin{cases} -\frac{ \epsilon}{2} (x^t_i-x^t_j) - \frac{\epsilon^2}{4} , &\quad  x^t_i-x^t_j \leq -\epsilon\\
\frac{ \epsilon}{2} (x^t_i-x^t_j) - \frac{\epsilon^2}{4}, & \quad x^t_j-x^t_i \leq -\epsilon,\\
0, & \quad \text{otherwise.}
\end{cases}
\]

This implies that \[D(y^{t+1})-D(y^t) \begin{cases}\geq \frac{\epsilon^2}{4}, & \qquad \text{if}\;  \Delta^t_e = 1,\\
 = 0, & \qquad \text{if}\; \Delta^t_e = 0.
\end{cases} \]

Taking expectation in the selection of $e$, we get 
\[\E{D(y^{t+1})-D(y^t) \;|\; y^t} \geq \frac{\epsilon^2}{4} \cdot \Prob(\Delta^t_e = 1 \;|\; y^t) + 0  \cdot \Prob(\Delta^t_e = 0 \;|\; y^t)=\frac{\epsilon^2}{4} \Delta^t.\]
It remains to take expectation again.

\subsection{Proof of Theorem~\ref{thm:09y09s9ffs}}

Since for all $k\geq 0$ we have $D(y^k) \leq D(y^*)$, it follows that \[D(y^*) - D(y^0) \geq \E{D(y^k)-D(y^0)} = \E{\sum_{t=0}^{k-1}D(y^{t+1}) - D(y^t)}= \sum_{t=0}^{k-1}\E{D(y^{t+1}) - D(y^t)}.\]
It remains to apply Lemma~\ref{lem:09ys09y9ss}.

\subsection{Proof of Lemma~\ref{L: noise exp iteration bound}}

Firstly we will compute increase in the dual function value  in iteration $t$:

\begin{eqnarray}
D(y^{t+1})-D(y^t)&=&\frac12 \|\overline{c}\ones-x^t \|^2-\frac12 \|\overline{c}\ones-x^{t+1} \|^2
\nonumber
\\
&=&
\frac12 \left( 
(\overline{c}-x^t_j)^2+(\overline{c}-x^t_i)^2-(\overline{c}-x^{t+1}_j)^2-(\overline{c}-x^{t+1}_i)^2
\right)
\nonumber
\\
&=&
-\overline{c}\left(
 x^{t+1}_j+x^{t+1}_i-x^{t}_j-x^{t}_i
\right)
+\frac12 \left(
\left(x^{t}_j\right)^2+\left(x^{t}_i\right)^2
-\left(x^{t+1}_j\right)^2-\left(x^{t+1}_i\right)^2
\right)
\nonumber
\\
&=&
-\overline{c}\left( w_j^{t_j}+w_i^{t_i} \right)
+\frac12 
\left(
\left(x^{t}_j\right)^2+\left(x^{t}_i\right)^2
\right)
\nonumber
\\
&& \qquad 
-\frac14 \left( 
\left(x_j^t+x_i^t  \right)^2
+ 
2\left(x_j^t+x_i^t \right)\left( w_j^{t_j}+w_i^{t_i} \right)
+\left( w_j^{t_j}+w_i^{t_i} \right)^2
\right)
\nonumber
\\
&=&
\frac14 \left(x_j^t-x_i^t\right)^2-
\left(w_i^{t_i}+w_j^{t_j}\right)
\left(    
\overline{c}+\frac12\left(x_j^t+x_i^t\right)
\right)
-\frac14\left(w_i^{t_i}+w_j^{t_j}\right)^2.
\label{Eq: noise_no_exp}
\end{eqnarray}

Now we want to estimate the expectation of this gap. Our main goal is to find $\E{D(y^{t+1}) - D(y^t)}$. There are $3$ terms in \eqref{Eq: noise_no_exp}. Since expectation is linear, we will evaluate expectations of the 3 terms separately and merge them at the end. 

Taking the expectation over the choice of edge and inserted noise in iteration $t$ we obtain
\begin{equation}
\E{\frac14\left(x_i^t - x_j^t\right)^2|x^t} = \frac{1}{4m}\sum\limits_{e\in \cE}\left(x_i^t-x_j^t\right)^2.\label{Eq: first_term}
\end{equation}

Thus we have

\begin{eqnarray*}
&&
\E{D(y^{t+1})-D(y^{t}) - \frac{1}{4}\left(x_i^t-x_j^t\right)^2 \;|\; x^t}
\\
&&
\qquad
\stackrel{\eqref{Eq: first_term}}{=}
\E{D(y^{t+1})-D(y^{t}) \;|\; x^t} - \frac{1}{4m}\sum\limits_{e\in \cE}\left(x_i^t-x_j^t\right)^2
\\
&&
\qquad
=
\E{D(y^*)-D(y^{t})\;|\; x^t} +\E{D(y^{t+1})-D(y^*)\;|\; x^t}  - \frac{1}{4m}\sum\limits_{e\in \cE}\left(x_i^t-x_j^t\right)^2
\\
&&
\qquad
\stackrel{\eqref{eq:99d8gds}}{=}
\frac12\|\overline{c}\ones-x^t \|^2
-\E{\frac12\|\overline{c}\ones-x^{t+1} \|^2 \;|\; x^t}
- \frac{1}{4m}\sum\limits_{e=(i,j)\in \cE}\left(x_i^t-x_j^t\right)^2
\\
&&
\qquad
\stackrel{\eqref{eq:8998gd98gikd}}{\leq}
\frac12\|\overline{c}\ones-x^t \|^2
-\E{\frac12\|\overline{c}\ones-x^{t+1} \|^2 \;|\; x^t}
- \frac{\ac(\cG)}{4m}\|\overline{c}\ones-x^t \|^2
\\
&&
\qquad
=
\left(1-\frac{\ac(\cG)}{2m}\right)\frac12\|\overline{c}\ones-x^t \|^2
-\E{\frac12\|\overline{c}\ones-x^{t+1} \|^2 \;|\; x^t}
\\
&&
\qquad
\stackrel{\eqref{eq:99d8gds}}{=}
\left(1-\frac{\ac(\cG)}{2m}\right)\left( D(y^*)- D(y^{t}) \right)
-\E{D(y^*)- D(y^{t+1})  \;|\; y^t}.
\end{eqnarray*}

Taking the full expectation of the above and using tower property, we get
\begin{equation}
\label{Eq: noisy gossip first full}
\E{D(y^{t+1})-D(y^{t}) - \frac{1}{4}\left(x_i^t-x_j^t\right)^2}
\leq
\left( 1-\frac{\ac(\cG)}{2m}\right)\E{D(y^*)- D(y^{t}) }
-\E{ D(y^*)- D(y^{t+1}) }.
\end{equation}

\begin{lem}\label{Lm: independence}
Suppose that we run Algorithm \algN\ for $t$ iterations and $t_i$ denotes the number of times that some edge corresponding to node $i$ was selected during the algorithm.  
\begin{enumerate}
\item  $v_i^{t_i}$ and $t_j$ are independent  for all (i.e., not necessarily distinct) $i,j$.
\item $v_i^{t_i}$ and $\phi_j^{t_j}$ are independent for all (i.e., not necessarily distinct)  $i,j$.
\item $w_i^{t_i}$ and $w_j^{t_j}$ have zero correlation for all $i\neq j$.
\item $x_j^t$ and $\phi_i^{t_i}v_i^{t_i}$ have zero correlation for all (i.e., not necessarily distinct) $i,j$.
\end{enumerate}
\end{lem}
\begin{proof}
\begin{enumerate}
\item Follows from the definition of $v_i^{t}$.
\item Follows from the definition of $v_i^{t}$.
\item Note that we have $w_i^{t_i} = \phi_i^{t_i}v_i^{t_i} - \phi_i^{t_i-1}v_i^{t_i-1}$ and $w_j^{t_j} = \phi_j^{t_j}v_j^{t_j} - \phi_j^{t_j-1}v_j^{t_j-1}$. 
Clearly, $v_i^{t_i}$ and $w_j^{t_j}$ have zero correlation. Similarly $v_i^{t_i-1}$ and $w_j^{t_j}$ have zero correlation. Thus, $w_i^{t_i}$ and $w_j^{t_j}$ have zero correlation.
\item Clearly, $x_j^t$ is a function initial state and all instances of random variables up to the iteration $t$. Thus, $v_i^{t_i}$ is independent to $x_j^t$ from the definition. Thus, $x_j^t$ and $\phi_i^{t_i}v_i^{t_i}$ have zero correlation.
\end{enumerate}
\end{proof}

Now we are going to take the expectation of the second term of \eqref{Eq: noise_no_exp}. We will use the ``tower rule" of expectations  in the form
$\E{\E{\E{X\;| \; Y,Z}\;|\;Y}} = \E{X}$,
where $X,Y,Z$ are random variables. In particular, we get
\begin{eqnarray*}
\E{-\left(w_i^{t_i}+w_j^{t_j}\right)
\left(    
\overline{c}+\frac12\left(x_j^t+x_i^t\right)
\right)} = \E{\E{\E{-\left(w_i^{t_i}+w_j^{t_j}\right)
\left(    
\overline{c}+\frac12\left(x_j^t+x_i^t\right)
\right) \;| \; e^t,x^t} \;| \; x^t}}.
\end{eqnarray*}

In the equation above, $e^t$ denotes an edge selected at in the iteration $t$.
Let us first calculate the inner most expectation on the right hand side of the above identity:
\begin{eqnarray*}
&&
\E{-\left(w_i^{t_i}+w_j^{t_j}\right)
\left(    
\overline{c}+\frac12\left(x_j^t+x_i^t\right)
\right)\;|\; e^t,x^t}
\\
&&
\qquad
\stackrel{(*^1)}{=}
\E{\left(\phi_i^{t_i-1}v_i^{t_i-1}+
\phi_j^{t_j-1}v_j^{t_j-1}- \phi_i^{t_i}v_i^{t_i}-
\phi_j^{t_j}v_j^{t_j}
\right)\left(    
\overline{c}+\frac12\left(x_j^t+x_i^t\right)
\right)\;|\; e^t,x^t}
\\
&&
\qquad
\stackrel{(*^2)}{=}
\E{ \overbrace{\left(\phi_i^{t_i-1}v_i^{t_i-1}+
\phi_j^{t_j-1}v_j^{t_j-1}
\right)}^{\text{constant}}
 \overbrace{\left(    
\overline{c}+\frac12\left(x_j^t+x_i^t
\right)
\right)}^{\text{constant}}\;|\; e^t,x^t}
\\
&&
\qquad \qquad
+\E{\left(- \phi_i^{t_i}v_i^{t_i}-
\phi_j^{t_j}v_j^{t_j}
\right)
\overbrace{\left(    
\overline{c}+\frac12\left(x_j^t+x_i^t\right)
\right)}^{\text{constant}}\;|\; e^t,x^t}\\
&&
\qquad
=
\left(\phi_i^{t_i-1}v_i^{t_i-1}+
\phi_j^{t_j-1}v_j^{t_j-1}
\right)\left(    
\overline{c}+\frac12\left(x_j^t+x_i^t\right)
\right)
\\
&&
\qquad
\qquad
+\cancelto{0}{\E{\left(- \phi_i^{t_i}v_i^{t_i}-
\phi_j^{t_j}v_j^{t_j}
\right)|e,x^t}}\left(    
\overline{c}+\frac12\left(x_j^t+x_i^t\right)
\right)
\\
&&
\qquad
\stackrel{L. \ref{Lm: independence}}{=}
\left(\phi_i^{t_i-1}v_i^{t_i-1}+
\phi_j^{t_j-1}v_j^{t_j-1}
\right)\left(    
\overline{c}+\frac12\left(x_j^t+x_i^t\right)
\right)
\\
&&
\qquad
=
\left(\phi_i^{t_i-1}v_i^{t_i-1}+
\phi_j^{t_j-1}v_j^{t_j-1}
\right)   
\overline{c}
+\frac12\left(\phi_i^{t_i-1}v_i^{t_i-1}+
\phi_j^{t_j-1}v_j^{t_j-1}
\right)  \left(x_j^t+x_i^t
\right),
\end{eqnarray*}  

where $(*^1)$ means definition of $w_i^{t_i}$ and $(*^2)$ means linearity of expectation.

Now we take the expectation of last expression above with respect to choice of edge on $t$-th iteration. We obtain

\begin{eqnarray*}
&&\E{\left(\phi_i^{t_i-1}v_i^{t_i-1}+
\phi_j^{t_j-1}v_j^{t_j-1}
\right)   
\overline{c}
+\frac12\left(\phi_i^{t_i-1}v_i^{t_i-1}+
\phi_j^{t_j-1}v_j^{t_j-1}
\right)  \left(x_j^t+x_i^t
\right)|x^t} 
\\
&&
\qquad
\stackrel{(*^2)}{=}
\frac12
\E{\left(\phi_i^{t_i-1}v_i^{t_i-1}+
\phi_j^{t_j-1}v_j^{t_j-1}
\right)  \left(x_j^t+x_i^t\right) |x^t} + \E{\cancelto{0}{\left(\phi_i^{t_i-1}v_i^{t_i-1}+\phi_j^{t_j-1}v_j^{t_j-1}
\right)  \overline{c} |x^t}}
\\
&&
\qquad
\stackrel{L. \ref{Lm: independence}}{=}
\frac12 \E{\left(\phi_i^{t_i-1}v_i^{t_i-1}+
\phi_j^{t_j-1}v_j^{t_j-1}
\right)  \left(x_j^t+x_i^t\right) |x^t}
\\
&&
\qquad
 \stackrel{(*^3)}{=}
\frac{1}{2m}\sum_{e\in \cE} \left(\phi_i^{t_i-1}v_i^{t_i-1}+
\phi_j^{t_j-1}v_j^{t_j-1}
\right)  \left(x_j^t+x_i^t\right) 
\\
&&
\qquad
=
\frac{1}{2m}\sum_{e\in \cE} \left( \phi_i^{t_i-1}v_i^{t_i-1} x_i^t +\phi_j^{t_j-1}v_j^{t_j-1} x_j^t\right)
+\frac{1}{2m} \sum_{e\in \cE} \left( \phi_i^{t_i-1}v_i^{t_i-1} x_j^t+ \phi_j^{t_j-1}v_j^{t_j-1} x_i^t
\right)
\\
&&
\qquad
\stackrel{(*^4)}{=}
\frac{1}{2m}\sum_{i=1}^n d_i \phi_i^{t_i-1}v_i^{t_i-1} x_i^t
+\frac{1}{2m} \sum_{e\in \cE} \left( \phi_i^{t_i-1}v_i^{t_i-1} x_j^t+ \phi_j^{t_j-1}v_j^{t_j-1} x_i^t
\right),
\end{eqnarray*}
where $(*^3)$ means definition of expectation and $(*^4)$ means change of the summation order.

\begin{lem}\label{Lm: another_independence}
\begin{equation}
\E{\phi_i^{t_i-1}v_i^{t_i-1}x_i^t} = \frac12\E{\left(\phi_i^{t_i-1}v_i^{t_i-1}\right)^2}.
\end{equation}
\end{lem}
\begin{proof}
\begin{eqnarray*}
&&
\E{\phi_i^{t_i-1}v_i^{t_i-1}x_i^t}
\\
&&
\qquad
 \stackrel{(*^5)}{=}
 \E{\phi_i^{t_i-1}v_i^{t_i-1}\left(\left(x_i^t - \frac{\phi_i^{t_i-1}v_i^{t_i-1}}{2}\right)+\frac{\phi_i^{t_i-1}v_i^{t_i-1}}{2}\right)} 
 \\
&&
\qquad
\stackrel{(*^2)}{=}
\E{\phi_i^{t_i-1}v_i^{t_i-1}\left(x_i^t - \frac{\phi_i^{t_i-1}v_i^{t_i-1}}{2}\right)}+\frac12\E{\left(\phi_i^{t_i-1}v_i^{t_i-1}\right)^2} 
\\
&&
\qquad
 \stackrel{(*^6)}{=}
 \E{\phi_i^{t_i-1}v_i^{t_i-1}\left(\frac{x_i^{t_i-1} + x_l^{t_l^0} +w^{t_i-1}+w^{t_l^0} -\phi_i^{t_i-1}v_i^{t_i-1}}{2}\right)}
 +
 \frac12\E{\left(\phi_i^{t_i-1}v_i^{t_i-1}\right)^2}
 \\
&&
\qquad
\stackrel{(*^1)}{=}
\E{\phi_i^{t_i-1}v_i^{t_i-1}\left(\frac{x_i^{t_i-1} + x_l^{t_l^0} +\phi_i^{t_i-1}v_i^{t_i-1}- \phi_i^{t_i-2}v_i^{t_i-2}+\phi_i^{t_l^0}v_l^{t_l^0} - \phi_i^{t_l^0-1}v_l^{t_l^0-1} -\phi_i^{t_i-1}v_i^{t_i-1}}{2}\right)}
\\
&&
\qquad \qquad
 \qquad +
\frac12\E{\left(\phi_i^{t_i-1}v_i^{t_i-1}\right)^2}
\\
&&
\qquad
=
\E{\phi_i^{t_i-1}v_i^{t_i-1}\left(\frac{x_i^{t_i-1} + x_l^{t_l^0} +\phi_i^{t_l^0}v_l^{t_l^0}- \phi_i^{t_l^0-1}v_l^{t_l^0-1}-\phi_i^{t_i-2}v_i^{t_i-2}}{2}\right)}+\frac12\E{\left(\phi_i^{t_i-1}v_i^{t_i-1}\right)^2}
\\
&&
\qquad
\stackrel{L.\ref{Lm: independence}}{=}\cancelto{0}{\E{\phi_i^{t_i-1}v_i^{t_i-1}}}\E{\left(\frac{x_i^{t_i-1} + x_l^{t_l^0} +\phi_i^{t_l^0}v_l^{t_l^0}-\phi_i^{t_l^0-1}v_l^{t_l^0-1}-\phi_i^{t_i-2}v_i^{t_i-2}}{2}\right)}
\\
&&
\qquad \qquad
 \qquad +
\frac12\E{\left(\phi_i^{t_i-1}v_i^{t_i-1}\right)^2} 
\\
&&
\qquad
=
\frac12\E{\left(\phi_i^{t_i-1}v_i^{t_i-1}\right)^2},
\end{eqnarray*}

where in step $(*^5)$ we add and subtracting $\frac{\phi_i^{t_i-1}v_i^{t_i-1}}{2}$. In the step $(*^6)$ we denote by $l$ a node such that that the noise $\phi_i^{t_i-1}v_i^{t_i-1}$ was added to the system when the edge $(i,l)$ was chosen (we do not consider $t_i=0$ since in this case the Lemma \ref{Lm: another_independence} trivially holds).

\end{proof}

Taking the expectation with respect to the algorithm we obtain

\begin{eqnarray}
&&
\E{-\left(w_i^{t_i}+w_j^{t_j}\right)
\left(   
\overline{c}+\frac12\left(x_j^t+x_i^t\right)
\right)} 
\nonumber
\\
&& 
\qquad 
\stackrel{(*^7)}{=}\E{\E{\E{-\left(w_i^{t_i}+w_j^{t_j}\right)\left(   
\overline{c}+\frac12\left(x_j^t+x_i^t\right)
\right) |x^t,e}|x^t}}
\nonumber
\\
&& 
\qquad 
=
\E{
 \frac{1}{2m}\sum_{i=1}^n d_i \phi_i^{t_i-1}v_i^{t_i-1} x_i^t
+ \frac{1}{2m}\sum_{e\in \cE} \left( \phi_i^{t_i-1}v_i^{t_i-1} x_j^t+ \phi_j^{t_j-1}v_j^{t_j-1} x_i^t
\right)
}
\nonumber
\\
&& 
\qquad
\stackrel{(*^2)}{=}
 \frac{1}{2m}\sum_{i=1}^n d_i \E{\phi_i^{t_i-1}v_i^{t_i-1} x_i^t}
+ \frac{1}{2m}\sum_{e\in \cE} \E{\left( \phi_i^{t_i-1}v_i^{t_i-1} x_j^t+ \phi_j^{t_j-1}v_j^{t_j-1} x_i^t
\right)
}
\nonumber
\\
&& 
\qquad
\stackrel{L.\ref{Lm: another_independence}}{=}
 \frac{1}{4m}\sum_{i=1}^n d_i \E{\left(\phi_i^{t_i-1}v_i^{t_i-1}\right)^2 }
+ \frac{1}{2m}\sum_{e\in \cE} \E{\left( \phi_i^{t_i-1}v_i^{t_i-1} x_j^t+ \phi_j^{t_j-1}v_j^{t_j-1} x_i^t
\right)
},
\label{Eq: noisy gossip second full}
\end{eqnarray}

where $(*^7)$ means tower rule.

\begin{lem}\label{Lm: squared_expetation}
\begin{equation}
\E{
\left(\phi_i^{t_i}v_i^{t_i}+
\phi_j^{t_j}v_j^{t_j}\right)^2
|x^t,e^t
} = \sigma^2_i\phi_i^{2t_i}+\sigma^2_j\phi_j^{2t_j}.
\label{Eq: squared expectation}
\end{equation}
\end{lem}
\begin{proof}
Since we have $\E{\left(\phi_i^{t_i}v_i^{t_i}+
\phi_j^{t_j}v_j^{t_j}\right)|x^t,e^t} = 0$, and also for any random variable $X$: $\E{X^2} = \Var{\left(X\right)}+\E{X}^2$, we only need to compute the variance:
\begin{eqnarray*}
\Var\left(\phi_i^{t_i}v_i^{t_i}+
\phi_j^{t_j}v_j^{t_j}\right)
 = 
 \Var{\left(\phi_i^{t_i}v_i^{t_i}\right)}+\Var{\left(
\phi_j^{t_j}v_j^{t_j}\right)} 
=
\left(\phi_i^{t_i}\right)^2\Var{\left(v_i^{t_i}\right)}+\left(\phi_j^{t_j}\right)^2\Var{\left(
v_j^{t_j}\right)} .
\end{eqnarray*}
\end{proof}

Taking an expectation of the third term of \eqref{Eq: noise_no_exp} with respect to lastly added noise, the expression $\E{\left(w_i^{t_i}+w_j^{t_j}\right)^2 |x^t,e^t}$ is equal to

\begin{eqnarray*}
& &
\stackrel{(*^1)}{=}
\E{\left(
\phi_i^{t_i}v_i^{t_i}+
\phi_j^{t_j}v_j^{t_j}
-
\phi_i^{t_i-1}v_i^{t_i-1}-
\phi_j^{t_j-1}v_j^{t_j-1}
\right)^2 |x^t,e^t}
\\
&&  
\stackrel{(*^2)}{=}
\E{
\left(\phi_i^{t_i}v_i^{t_i}+
\phi_j^{t_j}v_j^{t_j}\right)^2
|x^t,e^t
}
-
2
\E{\overbrace{
\left(\phi_i^{t_i}v_i^{t_i}+\phi_j^{t_j}v_j^{t_j}\right)
}^{\E{\phi_i^{t_i}v_i^{t_i}+\phi_j^{t_j}v_j^{t_j}}=0}
\overbrace{
\left(\phi_i^{t_i-1}v_i^{t_i-1}+\phi_j^{t_j-1}v_j^{t_j-1}\right)
}^{\text{constant}}
|x^t,e^t
}
\\
& & \qquad+
 \overbrace{\E{
\left(	
\phi_i^{t_i-1}v_i^{t_i-1}+
\phi_j^{t_j-1}v_j^{t_j-1}
\right)^2
|x^t,e^t}}^{\text{constant}}
\\
& & \qquad
\stackrel{\eqref{Eq: squared expectation}}{=}
\sigma_i^2\phi_i^{2t_i}+\sigma^2_j\phi_j^{2t_j}+\left(
\phi_i^{t_i-1}v_i^{t_i-1}+
\phi_j^{t_j-1}v_j^{t_j-1}
\right)^2.
\end{eqnarray*}
Taking the expectation over $e^t$ we obtain:

\begin{eqnarray*}
\E{\left(w_i^{t_i}+w_j^{t_j}\right)^2 |x^t}&\stackrel{(*^7)}{=}&\E{\E{\left(w_i^{t_i}+w_j^{t_j}\right)^2 |e^t, x^t}|x^t}
\\
&=&
\E{
\sigma_i^2\phi_i^{2t_i}+\sigma^2_j\phi_j^{2t_j}
+\left(
\phi_i^{t_i-1}v_i^{t_i-1}+
\phi_j^{t_j-1}v_j^{t_j-1}
\right)^2
|x^t}
\\
&\stackrel{(*^3)}{=}&
\frac{1}{m}\sum_{e \in \cE} \left( \sigma_i^2\phi_i^{2t_i}+\sigma^2_j\phi_j^{2t_j}\right)+
\frac1m \sum_{e\in \cE} \left(
\phi_i^{t_i-1}v_i^{t_i-1}+
\phi_j^{t_j-1}v_j^{t_j-1}
\right)^2
\\
&\stackrel{(*^4)}{=}&
\frac{1}{m}\sum_{i=1}^n d_i\sigma^2_i \phi_i^{2t_i}+
\frac1m \sum_{e\in \cE} \left(
\phi_i^{t_i-1}v_i^{t_i-1}+
\phi_j^{t_j-1}v_j^{t_j-1}
\right)^2.
\end{eqnarray*}
 
Finally, taking the expectation with respect to the algorithm we get

\begin{eqnarray}
&&\E{\left(w_i^{t_i}+w_j^{t_j}\right)^2 }
\nonumber
\\
&&
\qquad
\stackrel{(*^7)}{=}
\E{\E{\left(w_i^{t_i}+w_j^{t_j}\right)^2|x^t }}
\nonumber
\\
&&
\qquad
\stackrel{(*^2)}{=}
\frac{1}{m}\sum_{i=1}^n d_i\sigma^2_i \E{\phi_i^{2t_i}}+
\frac1m \sum_{e\in \cE}\E{ \left(
\phi_i^{t_i-1}v_i^{t_i-1}+
\phi_j^{t_j-1}v_j^{t_j-1}
\right)^2}
\nonumber
\\
&&
\qquad
=
\frac{1}{m}\sum_{i=1}^n d_i\sigma^2_i \E{\phi_i^{2t_i}}+
\frac1m \sum_{e\in \cE}\E{ 
\left(\phi_i^{t_i-1}v_i^{t_i-1}\right)^2+
\left(\phi_j^{t_j-1}v_j^{t_j-1}\right)^2
+ 2\phi_i^{t_i-1}v_i^{t_i-1}\phi_j^{t_j-1}v_j^{t_j-1}
}
\nonumber
\\
&&
\qquad
\stackrel{(*^4)}{=}
\frac{1}{m}\sum_{i=1}^n d_i\sigma^2_i \E{\phi_i^{2t_i}}+
\frac1m \sum_{i=1}^n d_i\E{ 
\left(\phi_i^{t_i-1}v_i^{t_i-1}\right)^2}
+ \frac2m \sum_{e\in \cE}\E{\phi_i^{t_i-1}v_i^{t_i-1}\phi_j^{t_j-1}v_j^{t_j-1}
}
\nonumber
\\
&&
\qquad
\stackrel{L.\ref{Lm: independence}}{=}
\frac{1}{m}\sum_{i=1}^n d_i\sigma^2_i \E{\phi_i^{2t_i}}+
\frac1m \sum_{i=1}^n d_i\E{ 
\left(\phi_i^{t_i-1}v_i^{t_i-1}\right)^2} + \frac2m \sum_{e\in \cE}\cancelto{0}{\E{\phi_i^{t_i-1}v_i^{t_i-1}}}\cancelto{0}{\E{\phi_j^{t_j-1}v_j^{t_j-1}}}
\nonumber
\\
&&
\qquad
=
\frac{1}{m}\sum_{i=1}^n d_i\sigma^2_i \E{\phi_i^{2t_i}}+
\frac1m \sum_{i=1}^n d_i\E{ 
\left(\phi_i^{t_i-1}v_i^{t_i-1}\right)^2
\label{Eq: noisy gossip third full}
}.
\end{eqnarray}

Combining \eqref{Eq: noise_no_exp} with \eqref{Eq: noisy gossip first full}, \eqref{Eq: noisy gossip second full} and \eqref{Eq: noisy gossip third full} we obtain

\begin{eqnarray*}
\E{ D(y^*)- D(y^{t+1}) } 
&\leq & 
\left( 1-\frac{\ac(\cG)}{2m}\right)\E{D(y^*)- D(y^{t})} 
-\frac{1}{4m}\sum_{i=1}^n d_i \E{\left(\phi_i^{t_i-1}v_i^{t_i-1}\right)^2 }
\\& &\qquad- \frac{1}{2m}\sum_{e\in \cE} \E{\left( \phi_i^{t_i-1}v_i^{t_i-1} x_j^t+ \phi_j^{t_j-1}v_j^{t_j-1} x_i^t
\right)}
\\
& & \qquad+\frac{1}{4m}\sum_{i=1}^n d_i \sigma^2_i \E{\phi_i^{2t_i}}+
\frac{1}{4m} \sum_{i=1}^n d_i\E{\left(\phi_i^{t_i-1}v_i^{t_i-1}\right)^2}
\\
&=&
\left( 1-\frac{\ac(\cG)}{2m}\right)\E{D(y^*)- D(y^{t})} +\frac{1}{4m}\sum_{i=1}^n d_i\sigma^2_i \E{\phi_i^{2t_i}}
\\& &\qquad-
\frac{1}{2m}\sum_{e\in \cE} \E{\left( \phi_i^{t_i-1}v_i^{t_i-1} x_j^t+ \phi_j^{t_j-1}v_j^{t_j-1} x_i^t
\right)},
\end{eqnarray*}
which concludes the proof.
\subsection{Proof of Theorem \ref{T: ng general convergence}}

\begin{lem}
After $t$ iterations of algorithm \algN\ we have
\begin{equation} \label{Eq: binomial}
\E{\phi_i^{2t_i}}=\left(1-\frac{d_i}{m} \left(1-\phi_i^2\right)\right)^{t}.
\end{equation}
\label{L: noisy binomial}
\end{lem}

\begin{proof}

\begin{eqnarray*}
\E{\phi_i^{2t_i}}&=&\sum_{j=0}^t \Prob(t_i=j)\phi_i^{2j}
\quad = \quad
\sum_{j=0}^t\binom{t}{j}\left(\frac{m-d_i}{m} \right)^{t-j}\left(\frac{d_i}{m}\phi_i^2\right)^j
\\
&=&
\left(\frac{m-d_i}{m} +\frac{d_i}{m}\phi_i^2\right)^{t}
\quad = \quad
\left(1-\frac{d_i}{m} \left(1-\phi_i^2\right)\right)^{t}.
\end{eqnarray*}

\end{proof}

\begin{lem}
Random variables $\phi_i^{t_i-1}v_i^{t_i-1} $ and $x_j^t$ are nonegatively correlated, i.e. 
\begin{equation}\label{Eq: positive corelation}
\E{\phi_i^{t_i-1}v_i^{t_i-1}x_j^t}\geq 0.
\end{equation}
\label{L: noisy positive correlation}
\end{lem}

\begin{proof}
Denote $R_{i,j}$ to be a random variable equal to 1 if the noise $w_i^{t_i}$ was added to the system when edge $(i,j)$ was chosen and equal to 0 otherwise. We can rewrite the expectation in the following way:

\begin{eqnarray*}
\E{\phi_i^{t_i-1}v_i^{t_i-1}x_j^t}&=&
\overbrace{\E{\phi_i^{t_i-1}v_i^{t_i-1}x_j^t \;|\; R_{i,j}=1}}^{\geq 0}\Prob(R_{i,j}=1)\\
\quad && +\overbrace{\E{\phi_i^{t_i-1}v_i^{t_i-1}x_j^t \;|\; R_{i,j}=0}}^{0}\Prob(R_{i,j}=0) \quad \geq \quad 0.
\end{eqnarray*}

The inequality  $
\E{\phi_i^{t_i-1}v_i^{t_i-1}x_j^t \;|\; R_{i,j}=1}\geq 0
$
holds due to the fact that $\phi_i^{t_i-1}v_i^{t_i-1}$ was added to $x_j$ with the positive sign.
\end{proof}

Combining \eqref{Eq: noise gossip iteration bound final} with Lemmas \ref{L: noisy binomial} and \ref{L: noisy positive correlation} we obtain

\begin{eqnarray*}
\E{ D(y^*)- D(y^{t+1}) } 
&\leq&
\left( 1-\frac{\ac(\cG)}{2m}\right)\E{D(y^*)- D(y^{t})} +\frac{1}{4m}\sum_{i=1}^n d_i \sigma^2_i \E{\phi_i^{2t_i}}
\\
& &\qquad-
\frac{1}{2m}\sum_{e\in \cE} \E{\left( \phi_i^{t_i-1}v_i^{t_i-1} x_j^t+ \phi_i^{t_j-1}v_j^{t_j-1} x_i^t
\right)}
\\
&\stackrel{\eqref{Eq: positive corelation}}{\leq} &
\left( 1-\frac{\ac(\cG)}{2m}\right)\E{D(y^*)- D(y^{t})} +\frac{1}{4m}\sum_{i=1}^n d_i\sigma^2_i  \E{\phi_i^{2t_i}}
\\
&\stackrel{\eqref{Eq: binomial}}{=} &
\left( 1-\frac{\ac(\cG)}{2m}\right)\E{D(y^*)- D(y^{t})} +\frac{1}{4m}\sum_{i=1}^n d_i\sigma^2_i  \left(1-\frac{d_i}{m}\left(1-\phi_i^2\right) \right)^{t}
\\
&=&
\left( 1-\frac{\ac(\cG)}{2m}\right)\E{D(y^*)- D(y^{t})} +\frac{\sum\left(d_i\sigma_i^2\right)}{4m}\psi^t.
\end{eqnarray*}

The recursion above gives us inductively the following

\begin{equation*}
\E{ D(y^*)- D(y^{k}) } \leq  \rho^k \left( D(y^*)- D(y^{0}) \right)  + \frac{\sum\left(d_i\sigma_i^2\right)}{4m}\sum_{t=1}^k \rho^{k-t}\psi^{t},
\end{equation*}
which concludes the proof of the theorem.

\subsection{Proof of Corollary \ref{C: noisy gossip special}}

Note that we have
\begin{eqnarray*}
\psi^t
&=&\frac{1}{\sum_{i=1}^n d_i\sigma_i^2 }
\sum_{i=1}^n d_i \sigma^2_i \left(1-\frac{d_i}{m}\left(1-\left(1-\frac{\gamma}{d_i}\right)\right) \right)^{t}\\
&=&
\frac{1}{\sum_{i=1}^n d_i\sigma_i^2 }
\sum_{i=1}^n d_i \sigma^2_i\left(1-\frac{\gamma}{m}\right)^{t}
\quad = \quad \left(1-\frac{\gamma}{m}\right)^{t}.
\end{eqnarray*}

In view of Theorem~\ref{T: ng general convergence}, this gives us the following:
\begin{eqnarray*}
\E{ D(y^*)- D(y^{k}) } 
&\leq &
 \left( D(y^*)- D(y^{0}) \right) \rho^k + \frac{\sum\left(d_i\sigma_i^2\right)}{4m}\sum_{t=1}^k \rho^{k-t}\psi^{t}
\\
&\leq & 
\left( D(y^*)- D(y^{0}) \right) \rho^k + \frac{\sum\left(d_i\sigma_i^2\right)}{4m}\sum_{t=1}^k \rho^{k-t}\left(1-\frac{\gamma}{m}\right)^{t}
\\
&\leq &
\left( D(y^*)- D(y^{0}) \right) \rho^k + \frac{\sum\left(d_i\sigma_i^2\right)}{4m}k \max\left(\rho,1-\frac{\gamma}{m}\right)^k
\\
&\leq  &
\left( D(y^*)- D(y^{0})+\frac{\sum\left(d_i\sigma_i^2\right)}{4m}k \right) \max\left(\rho,1-\frac{\gamma}{m}\right)^k.
\end{eqnarray*}

\section{Notation Glossary}

The following notational conventions are used throughout the paper. Boldface upper case letters will denote matrices and boldface lower case letters will denote vectors.

\begin{table}[!h]
\begin{center}
\begin{tabular}{|c|l|c|}
 \hline
 \multicolumn{2}{|c|}{{\bf Graphs} }\\
 \hline
$\cG = (\cV,\cE)$ & an undirected graph with vertices $\cV$ and edges $\cE$ \\
$n$ & $=|\cV|$ (number of vertices)\\
$m$ & $=|\cE|$ (number of edges) \\
$e = (i,j) \in \cE$ & edge of $\cG$ connecting nodes $i,j\in \cV$  \\
$d_i$ & degree of node $i$  \\
$c \in \R^n$ & $=(c_1,\dots,c_n)$; a vector of private values stored at the nodes of $\cG$  \\
$\bar{c}$ &$\bar{c}=\tfrac{1}{n}\sum_i c_i$ (the average of the private values)  \\
$\bA \in \R^{m\times n}$ &  \\
$\bL \in \R^{m\times m}$ & $=\bA \bA^\top$ (Laplacian of $\cG$)  \\
$\ac(\cG)$ & $=\lambda_{\min}^+(\bL)$ (algebraic connectivity of $\cG$)  \\
$\beta(\cG)$ & $=n/\ac(\cG)$  \\

\hline
\multicolumn{2}{|c|}{{\bf Randomness}}\\
\hline
$\Exp$ & expectation  \\
$\Prob$ & probability  \\
$\Var$ & variance  \\
$v_i^k$ & random variable from $N(0,\sigma_i^2)$   \\

\hline
\multicolumn{2}{|c|}{{\bf Optimization} }\\
\hline
$P: \R^n \to \R$ & primal objective function   \\
$D: \R^m \to \R$ & dual objective function  (a concave quadratic)  \\
$y \in \R^m$ & dual variable  \\
$y^* \in \R^m$ & optimal dual variable  \\
$x \in \R^n$ & primal variable  \\
$x^* \in \R^n$ & $=\bar{c}\ones$ (optimal primal variable)  \\
$\ones$ & a vector of all ones in $\R^n$  \\

\hline
\multicolumn{2}{|c|}{{\bf Summation} }\\
\hline

$\sum_{i}\sum_j$ & sum through all ordered pairs of $i$ and $j$ \\
$\sum_{(i,j)}$ & sum through all unordered pairs of $i$ and $j$ \\
$\sum_{(i,j) \in \cE}$ & sum through all edges of $\cG$ \\

\hline
\end{tabular}
\end{center}
\caption{The main notation used in the paper.}
\label{tbl:notation}
\end{table}

\end{document}